\documentclass[11pt]{amsart}

\usepackage{amsmath}
\usepackage{amssymb}

\newtheorem{theorem}{\bf Theorem}[section]
\newtheorem{assumption}{\bf Assumption}[section]
\newtheorem{corollary}[theorem]{\bf Corollary}
\newtheorem{definition}{\bf Definition}[section]
\newtheorem{lemma}{\bf Lemma}[section]
\newtheorem{problem}{\bf Problem}[section]
\newtheorem{proposition}{\bf Proposition}[section]

\newtheorem{conjecture}{Conjecture}[section]

\newtheorem{example}{Example}[section]

\newtheorem{remark}{Remark}[section]

\begin{document}
\title[Generalized Minimizers for Control-Affine Problems]{Measuring
Singularity of Generalized Minimizers for Control-Affine Problems }
\author[Manuel Guerra]{Manuel Guerra$^1$}
\author[Andrey Sarychev]{Andrey Sarychev$^2$}
\address{$^1$Instituto Superior de Economia e Gest\~ao, Technical University
of Lisbon, R. do Quelhas 6, 1200-781 Lisboa, Portugal,
mguerra@iseg.utl.pt
\\
$^2$Dipartimento di Matematica per le Decisioni, Universit\`a di
Firenze, via C.Lombroso 6/17, 50134 - Firenze (FI), Italy,
asarychev@unifi.it}
\thanks{The first author has been partially supported by Funda\c{c}\~{a}o
para a Ci\^{e}ncia e a Tecnologia (FCT), Portugal, co-financed by
the European Community Fund FEDER/POCI via Research Center on
Optimization and Control (CEOC) of the University of Aveiro,
Portugal. The second author has been partially supported by MIUR,
Italy via PRIN 2006019927}

\begin{abstract}
An open question contributed by Yu. Orlov to a recently published
volume "Unsolved Problems in Mathematical Systems and Control
Theory", V.D. Blondel \& A. Megretski (eds), Princeton Univ. Press,
2004, concerns regularization of optimal control-affine problems.
These noncoercive problems in general admit 'cheap (generalized)
controls' as minimizers; it has been questioned whether and under
what conditions infima of the regularized problems converge to the
infimum of the original problem. Starting with a study of this
question we show by simple functional-theoretic reasoning that it
admits, in general, positive answer. This answer does not depend on
com\-mu\-ta\-ti\-vity/non\-com\-mu\-ta\-ti\-vity of controlled
vector fields. It depends instead on presence or absence of a
Lavrentiev gap.

We set an alternative question of measuring "singularity" of
minimizing sequences for control-affine optimal control problems by
so-called degree of singularity. It is shown that, in the particular
case of singular linear-quadratic problems, this degree is tightly
related to the "order of singularity" of the problem. We formulate a
similar question for nonlinear control-affine problem and establish
partial results. Some conjectures and open questions are formulated.
\end{abstract}

\maketitle

\thispagestyle{empty} 



\section{Introduction \label{S0}}

The following open question, suggested by Yu. Orlov, appeared in a recently
published volume by V. Blondel et. al. \cite{Orl}.

Consider an optimal control problem.
\begin{eqnarray}
&&J_{0}^{T}(u(\cdot ))=\int_{0}^{T}x(t)^{\prime }Px(t)dt\rightarrow \min ,
\label{j0t} \\
&&\dot{x}=f(x)+G(x)u,\qquad x(0)=x_{0}.  \label{af}
\end{eqnarray}%
$T\in ]0,+\infty ]$ is fixed, $P$ denotes a symmetric definite positive
matrix, $f$ is a smooth vector field and $G(x)=\left(
g_{1}(x),g_{2}(x),...,g_{k}(x)\right) $ is an array of smooth vector fields.
An endpoint condition
\begin{equation}
x(T)=x_{T}  \label{ep}
\end{equation}%
can be added when $T<\infty $.

Consider a regularization of this problem, which amounts to minimization of
the penalized functional
\begin{equation}  \label{jet}
J_\varepsilon^{T}(u(\cdot))= \int_0^{T}x(t)^{\prime}Px(t)
+\varepsilon |u(t)|^2 dt \rightarrow \min,
\end{equation}
calculated along the trajectories of (\ref{af}). The question put by Yu.
Orlov in \cite{Orl} was whether and under what assumptions
\begin{equation*}
\lim\limits_{\varepsilon \rightarrow 0^+} \min\limits_u
J_\varepsilon^{T}(u)=\inf\limits_u J_0^{T}(u).
\end{equation*}
We show that the answer to this question is positive in almost all
cases. Further, the result holds for every nonnegative penalization
(not necessarily quadratic) that one may chose to regularize the
functional (\ref{j0t}).

We suggest an alternative question. In our opinion it is not the \textit{%
values} of the infima which should be studied, but rather the asymptotics of
the regularized functionals along minimizing sequences of $J_{0}^{T}$.
Indeed, it is quite general phenomenon that, for generic data $%
\lim\limits_{m \rightarrow \infty } \left\Vert u^{(m)}\right\Vert
_{L_{2}}=+\infty $ holds for any sequence $\left\{ u^{(m)}\in L_{\infty
,loc},\,m\in \mathbb{N}\right\} $ such that $\lim\limits_{m \rightarrow
\infty } J_{0}^{T}(u^{(m)})=\inf\limits_{u}J_{0}^{T}(u)$.

We trust that the minimal rate of growth of a sequence $\left\{ \left\Vert
u^{(m)}\right\Vert _{L_{2}},\, m \in \mathbb{N }\right\} $ that can be
achieved when $J_{0}^{T}(u^{(m)})\leq \inf\limits_{u}J_{0}^{T}(u)+\frac{1}{m}
$, is an important property characterizing the degree of singularity of
problem (\ref{j0t})-(\ref{af}) or (\ref{j0t})-(\ref{af})-(\ref{ep}).

For the particular case of singular linear-quadratic problems we are able to
fully characterize all types of singularities that occur. For nonlinear
control-affine case we provide partial answers.

This paper is organized as follows. In Section~\ref{S1} we answer
Yu. Orlov's question in finite-horizon and infinite-horizon settings
and introduce an extension of this question demonstrating its
interrelation with Lavrentiev phenomenon in calculus of variations
and optimal control. In Section \ref{ds} we introduce the notion of
"degree of singularity" (Definition~\ref{degsing}) and set the
problem of 'measuring singularity' of generalized minimizers. In
Section \ref{S2} we give a full characterization of possible values
of the degree of singularity for singular linear-quadratic problem
in finite-horizon and infinite-horizon settings (Theorems~\ref{c
muT}, \ref{strat}, \ref{c muInf}). In Section~\ref{cap} we introduce
the case of nonlinear control-affine problems
(\ref{j0t})-(\ref{af})-(\ref{ep}). The general driftless case is
solved in Section \ref{ncdr} (Theorem \ref{Thdriftless}). In Section
\ref{cafm} we provide an upper estimate of the degree of singularity
for the case, where the cost is positive state-quadratic and
controlled vector fields $g_{i}$ commute: $\left[ g_{i},g_{j}\right]
=0$, $\forall i,j$ (Theorem~\ref{taf3/2}). In Section~\ref{S8} we
provide some evidence for existence of a better estimate for the
commutative case and illustrate by example. The proofs of several
results discussed in Sections \ref{S2}-\ref{cafm} are quite
technical and are collected in Section~\ref{S4}.

A brief exposition of part of these results has appeared in
\cite{GS1}.

Partial results for generic control-affine problems
(\ref{j0t})-(\ref{af})-(\ref{ep}) with noncommuting inputs will be
the object of a separate publication.

We are grateful to an anonymous referee who brought to our attention
some additional bibliographic references and whose stimulating
questions and remarks allowed us to (hopefully) improve the
presentation especially in what regards Section~\ref{S1}.

\section{Convergence of regularized functionals and Lavrentiev phenomenon\label{S1}}

In this Section we answer Yu. Orlov's question in a slightly more
general setting. Namely, we will consider in place of (\ref{jet}),
the functional penalized by $\varepsilon \rho (t,u(t))$
\begin{equation}
\label{jetg} J_\varepsilon^T \left( u(\cdot ) \right) = \int_0^T
x(t)^\prime P x(t) + \varepsilon \rho (t,u(t)) \, dt,
\end{equation}
where $\rho : [0,T] \times \mathbb R^k \mapsto [0,+\infty[ $
is a nonnegative Borel function. We denote by
$\mathcal U_\rho$ the set of admissible controls for the problem
(\ref{jetg})-(\ref{af}) ((\ref{jetg})-(\ref{af})-(\ref{ep})):
\begin{equation}\label{urho}
\mathcal U_\rho = \left\{ u:[0,T] \mapsto \mathbb R^k \, | \, u \
{\rm is \ measurable,} \int_0^T \rho (t,u(t)) \, dt < + \infty
\right\} ,
\end{equation}
provided we set $J_\varepsilon^T(u) = + \infty$ for any $u \in
\mathcal U_\rho$ for which (\ref{af}) does not admit solution in the
interval $[0,T[$.
It is clear that $\mathcal U_\rho = L_p^k[0,T]$, whenever
$\rho(t,u) = |u|^p$ in \eqref{urho} with $p \in [1, + \infty [$,
as it is often the case.

Basically, a {\it positive} answer to Yu.~Orlov's question is contained
in the following result.
\smallskip

\begin{theorem}
\label{pos}Let $\mathcal U_\rho$ be a class of admissible controls
defined by \eqref{urho}, $J^T_0(u), \ J^T_\varepsilon(u)$ be the original cost
functional \eqref{j0t}
and the regularized cost functional \eqref{jetg}, respectively. Then
\begin{equation}
 \label{lit}  \lim\limits_{\varepsilon \rightarrow
0^+}\,\inf\limits_{u \in \mathcal U _\rho}
J^T_\varepsilon(u)=\inf_{u \in \mathcal U _\rho} J^T_0(u) . \
\square
\end{equation}
\end{theorem}

\begin{proof}
Fix $u^{(m)} \in \mathcal U _\rho$, a minimizing sequence for
$J^T_0$. Without loss of generality one may think that
$J^T_0(u^{(m)}) \leq \inf J^T_0 +1/m.$ Let $\int_0^T \rho(t,u(t)) \,
dt =\nu_m$. Then,
\begin{align*}
\inf J^T_0 \leq & \inf J^T_\varepsilon \leq J^T_\varepsilon(u^{(m)})
=
\\
= & J^T_0(u^{(m)}) + \varepsilon\nu_m \leq \inf J^T_0 +1/m+
\varepsilon\nu_m.
\end{align*}
Taking $\varepsilon_m=\nu_m^{-1}/m$ we conclude that
\begin{equation*}
\inf J^T_0 \leq \inf J^T_{\varepsilon_m} \leq J^T_{\varepsilon_m}(u^{(m)})
\leq \inf J^T_0 +2/m
\end{equation*}
and hence \eqref{lit} holds.
\end{proof}

Theorem \ref{pos} has the following immediate Corollary: \smallskip

\begin{corollary}
If $u_\varepsilon(\cdot)  \in \mathcal U _\rho $ are minimizers of
the regularized problems (\ref{jetg})-(\ref{af}) or (\ref{jetg})-(\ref{af})-(\ref{ep}), then
\begin{equation*}
\lim\limits_{\varepsilon \rightarrow
0^+}J_0^T(u_\varepsilon)=\inf_{u \in \mathcal U _\rho} J_0^T . \
\square
\end{equation*}
\end{corollary}

Note that the main point of {\it regularizing} a singular problem
(\ref{j0t})-(\ref{af}) or (\ref{j0t})-(\ref{af})-(\ref{ep}) is to
obtain a similar problem possessing a minimizer in a suitable space
of {\it regular} controls. Existence results for optimal control
problems with control-affine dynamics typically require superlinear
growth of the integrand in the cost functional, as $|u| \rightarrow
\infty$ (see e.g. \cite{Ces}), i.e. $\rho(t,|u|)/|u| \rightarrow
\infty$ as $|u| \rightarrow \infty$, uniformly with respect to $t
\in [0,T]$. Therefore, a penalty of type $\rho(t,u)=|u|^{1+\eta}$
($\eta
> 0$, constant) typically guarantees existence of solution for the
regularized problem, while a penalty of type $\rho(t,u)=|u|$ may
fail to do it. It is also natural to assume that all the  classes
$\mathcal U _\rho$ of controls (see \eqref{urho}) are contained in
$L_1^k[0,T]$; otherwise it is hard to verify existence of
trajectories of the control-affine system \eqref{af}.

An interesting extension of the original question formulated in
\cite{Orl} would be to admit not only the possibility of regularizing the
functional but also of 'regularizing' its domain of definition.

This would mean introducing {\it two classes of controls} $\mathcal
U \supset \mathcal U _\rho$, considering the functional \eqref{j0t}
on $\mathcal U$ while considering the regularized functionals
\eqref{jetg} in a smaller class of 'more regular' controls $\mathcal
U _\rho$. Here $\mathcal U $ can be any suitable class of controls,
not necessarily defined by an equality of type (\ref{urho}).

Our {\it extended question} would be whether
\begin{equation}\label{rho0}
\lim\limits_{\varepsilon \rightarrow 0^+}\,\inf\limits_{u \in
\mathcal U _\rho} J^T_\varepsilon(u)=\inf_{u \in \mathcal U }
J^T_0(u) \ ?
\end{equation}

This question  turns out to be tightly related (and in fact {\it
equivalent}) to another prominent issue of the calculus of
variations and optimal control - the {\it Lavrentiev phenomenon}
(see \cite{Ces} for a brief account and historical remarks).

Recall that a functional $J^T_0(u)$ defined on a class $\mathcal U
\supset \mathcal U _\rho$ exhibits {\it Lavrentiev phenomenon}  or
possesses {\it $\mathcal U -\mathcal U_\rho$ Lavrentiev gap} if
\begin{equation}\label{gap}
 \inf_{u \in \mathcal U } J^T_0(u) < \inf_{u \in \mathcal U
_{\rho}} J^T_0(u) .
\end{equation}

The following elementary result shows that validity of \eqref{rho0}
is equivalent to nonoccurence of Lavrentiev phenomenon for $J^T_0$.

\begin{theorem}\label{lavr} Equality \eqref{rho0} holds if and only if
$J^T_0$ does not possess  $\mathcal U -\mathcal U_\rho$ Lavrentiev
gap. $\square$
\end{theorem}
\begin{proof}
Whenever we have equality in place of strict inequality in
\eqref{gap} there exists a minimizing sequence $u^{(m)} \in \mathcal
U _\rho$ such that $\lim\limits_{m \rightarrow \infty}
J^T_0(u^{(m)}) = \inf\limits_{u \in \mathcal U } J^T_0(u).$ Now
\eqref{rho0} is concluded in the same way as \eqref{lit} has been
concluded in the proof of the Theorem~\ref{pos}.

Note that by Theorem~\ref{pos} $\lim\limits_{\varepsilon \rightarrow
0^+}\,\inf\limits_{u \in \mathcal U _\rho} J^T_\varepsilon(u) =
\inf\limits_{u \in \mathcal U _\rho} J^T_0(u)$. By direct
computation $\lim\limits_{\varepsilon \rightarrow
0^+}\,\inf\limits_{u \in \mathcal U _\rho} J^T_\varepsilon(u) \geq
\inf\limits_{u \in \mathcal U } J^T_0(u)$. Whenever the last
inequality  is strict we immediately conclude the presence of
Lavrentiev gap $\inf\limits_{u \in \mathcal U _\rho} J^T_0(u) >
\inf\limits_{u \in \mathcal U } J^T_0(u). \ \square$
\end{proof}

Thus we have completely reduced the validity of the equality
\eqref{rho0} above to the nonoccurence of $\mathcal U -\mathcal
U_\rho$ Lavrentiev gap for the optimal control problem
\eqref{j0t}-\eqref{af} or \eqref{j0t}-\eqref{af}-\eqref{ep}.

Lavrentiev phenomenon has been  mainly studied for the classical
problem of Calculus of Variations, Some partial results
regarding occurrence of this phenomenon for optimal control problems
are known; see \cite{CM,Sarl} where there are examples of Lavrentiev
phenomenon occurring for variational problems with higher-order
derivatives; these problems can be interpreted as Lagrange problems
with linear dynamics.

The Lavrentiev phenomenon is seen more as a rarity; the above cited
results certainly involve more sophisticated cost functionals than
the quadratic functional \eqref{j0t}, though the dynamics involved are
linear autonomous in contrast to \eqref{af}.

In the case of finite horizon $T< + \infty $, nonoccurrence of
$L_1^k[0,T]-L_\infty^k[0,T]$ Lavrentiev gap in
\eqref{j0t}-\eqref{af} or \eqref{j0t}-\eqref{af}-\eqref{ep} can be
easily proved. To see this, consider a minimizing sequence $\left\{
u^{(m)} \in L_1^k[0,T] , m \in \mathbb N \right\} $ of the
functional $J^T_0(u)$. Recall that the input/trajectory mapping
$u(\cdot) \mapsto x_u(\cdot)$ is continuous (on some $L_1^k[0,T]
$-neighborhood of each $u ^{(m)} $) with respect to $L_1 ^k [0,T]$
metric of $u$'s and $L_\infty ^n [0,T]$-metric of $x_u$'s. Then the
map $u \mapsto J^T_0(u)=\int_0^Tx_u'Px_udt$  is continuous. As we
know the functions from $ L^k_1[0,T]$ are approximable in
$L_1$-metric by functions from $L^k_\infty[0,T]$. Hence, taking
proper approximations of the functions $\left\{ u^{(m)}\right\}$ we
can construct for $J^T_0(u)$ a minimizing sequence $\left\{
\bar{u}^{(m)}\right\}$ of functions from $L^k_\infty[0,T]$.
Therefore $\inf\limits_{u \in L_\infty^k[0,T]} J^T_0(u) =
\inf\limits_{u \in L_1^k[0,T]} J^T_0(u) $, which implies equality in
\eqref{rho0} according to Theorem~\ref{lavr}. Thus we proved the
following

\begin{theorem}
\label{finor} Consider the problem \eqref{j0t}-\eqref{af}
(\eqref{j0t}-\eqref{af}-\eqref{ep}) with finite horizon $T< + \infty
$. Equality \eqref{rho0} holds for any classes of controls
$L_1^k[0,T] \supset \mathcal U \supset \mathcal U_\rho \supset
L_\infty^k[0,T]. \ \square$
\end{theorem}

We are not aware of any results on occurrence/nonoccurrence of
Lavrentiev phenomenon for infinite horizon.  We provide below
conditions which can be imposed on the control system \eqref{af} in
order to guarantee the lack of Lavrentiev gap for the problem
\eqref{j0t}-\eqref{af} with $T=+\infty$ and validity of equality
\eqref{rho0} for a pair $L^k_{1,loc} , \ L^k_p[0,+\infty[   $.

\begin{definition}
\label{detc} The control affine system (\ref{af}) is said to be
\textit{locally stabilizable of order $\alpha$} if there exists a
Lipschitzian feedback $\bar{u}(x) $, and a constant $C< + \infty$
such that $\bar{u}(0)=0$ and $ |x(t;x_0)| \leq
C|x_0|(t+1)^{-\alpha}$ holds for every $x_0$ in some neighborhood of
the origin. Here $x(t;x_0)$ is the trajectory of the ODE
\begin{equation*}
\dot{x}=f(x)+G(x)\bar u(x), \qquad x(0)=x_0. \ \square
\end{equation*}
\end{definition}

\begin{theorem}
\label{lav} Assume the horizon to be infinite: $T=+\infty$. If the
system (\ref{af}) is locally stabilizable of order $\alpha >
\frac{1}{p}$, with $p \in [1,2[$, then $J_0^\infty $ does not have
$L_{1,loc}-\left( L_p \cap L_\infty \right)$ Lavrentiev gap, i.e.
\[
\inf_{u \in L_p^k[0,+\infty[ \cap L_\infty^k[0,+\infty[ }J_0^\infty(u) = \inf_{u \in L_{1,loc}^k}J_0^\infty(u).
\]
For $p \in [2, + \infty]$ the equality holds provided $\alpha >
\frac 1 2$. $\square$
\end{theorem}

\begin{proof}
In the case $\inf\limits_{u \in L_{1,loc}^k} J_0^\infty (u) = +
\infty$, the theorem holds trivially. Hence we only need to consider
the case when $\inf\limits_{u \in L_{1,loc}^k} J_0^\infty (u) < +
\infty$.

Fix $\varepsilon >0$. For each $u \in L_{1,loc}^k$ let $x_u$ denote
the corresponding trajectory of system (\ref{af}). There exists
$\tilde u \in L_{1,loc}^k$ such that $\int_0^\infty x_{\tilde
u}(t)^\prime P x_{\tilde u}(t) dt < \inf\limits_{u \in L_{1,loc}^k}
J_0+\varepsilon < +\infty$. Then $\lim\limits_{t \rightarrow +
\infty} \int _t^{+ \infty} x_{\tilde u}(\tau)^\prime P x _{\tilde
u}(\tau) d\tau = 0$. Since $P$ is positive there must exist a
sequence $\left\{ t_j \right\} \rightarrow +\infty $ for which
$x_{\tilde u}(t_j) \rightarrow 0$. Since $u \mapsto x_u(\cdot )$ is
a continuous mapping from $L_1^k[0,t_j]$ into $L_\infty^n[0,t_j]$,
it follows by density of $L_\infty^k[0,t_j]$ that there exist
controls $u_j \in L_\infty^k[0,t_j]$ such that
\begin{align*}
\int_0^{t_j}x_{u_j}(t) ^\prime P x_{u_j}(t) \, dt & \leq
\int_0^{t_j}x_{\tilde u }(t) ^\prime P x_{\tilde u }(t) \, dt +
\varepsilon \leq \inf _{u \in L_{1,loc}^k} J_0^\infty (u) + 2
\varepsilon; \\
x_{u_j}(t_j) & \rightarrow 0.
\end{align*}
Concatenate the trajectory $x _{ u_j}(\cdot)$ with the trajectory
$y_j(t)$ starting at $x_{ u_j}(t_j)$ and driven by the feedback
control $\bar{u}$ from Definition \ref{detc}. For every sufficiently
large $j \in \mathbb{N}$ there holds
\begin{equation*}
|y_j(t)^{\prime}Py_j(t)| \leq C_1 |x_{u_j}(t_j)|^2(1+t-t_j)^{-2
\alpha}, \qquad \forall t\geq t_j.
\end{equation*}
This implies
\begin{equation*}
\int_{t_j} ^{+\infty}y_j(t)^{\prime}Py_j(t) dt \leq
\frac{C_1}{2\alpha -1 } |x_{u_j}(t_j)|^2 .
\end{equation*}
Since the feedback control $\bar{u}(x)$ is Lipschitzian and
$\bar{u}(0)=0$, then $| \bar{u}(y_j(t)) | \leq C_2 |y_j(t)| \leq C_3
| x_{u_j}(t_j)|(1+t-t_j)^{-\alpha}$, and hence for some $M<+\infty$,
\linebreak $\int_{t_j}^{+\infty} |\bar{u}(y_j(t))|^p dt < M$, holds
for all sufficiently large $j \in \mathbb{N}.$ This proves that the
control,
\begin{equation*}
\hat u_j (t)=\left\{
\begin{array}{ll} u_j (t) & if \ t \leq t_j; \\
\bar u (y_j(t)) & if \ t >t_j,%
\end{array}%
\right.
\end{equation*}
is of class $L_p^k[0,+\infty[ \cap L_\infty^k[0,+\infty[$. Evaluating the functional $J_0$
along the corresponding concatenated trajectory $x_{\hat
u_j}(\cdot)$ we conclude that
\begin{equation*}
J_0(\hat u_j ) \leq \inf_{u \in L_{1,loc}^k}
J_0+2\varepsilon+\frac{C_1}{2\alpha -1 }|x_{u_j}(t_j)|^2.
\end{equation*}
Choosing $j$ sufficiently large we thus construct a control
$u_\varepsilon (\cdot) \in L_p^k[0,+\infty[$ for which
\begin{equation*}
J_0(u_\varepsilon) \leq \inf_{u \in L_{1,loc}^k} J_0+ 3 \varepsilon
.
\end{equation*}
Taking $\varepsilon \rightarrow 0$ we arrive to a minimizing sequence of
$p$-integrable controls. The proof is completed by application of Theorem %
\ref{pos}.
\end{proof}

\smallskip

The following corollary follows immediately from Theorems \ref{lavr}
and  \ref{lav}:

\begin{corollary}
\label{P2} Assume the horizon to be infinite: $T=+\infty$. If the
system (\ref{af}) is locally stabilizable of order $\alpha >
\frac{1}{p}$, with $p \in [1,2[$, then
\[
\lim_{\varepsilon \rightarrow 0^+} \inf_{u \in L_p^k[0,+\infty[}
\left( J_0^\infty (u) + \varepsilon \| u\|_{L_p^k[0,+\infty[
}^p\right) = \inf_{u \in L_{1,loc}^k}J_0^\infty(u).
\]
For $p \in [2, + \infty[$ the equality holds provided $\alpha >
\frac 1 2$. $\square$
\end{corollary}

Note that the convergence issue for regularized functionals is
settled  by elementary functional-theoretic arguments   and the
answer does not depend on commutativity assumptions for the
controlled vector fields and other issues typically involved in the
study of generalized controls. We would like to formulate now a
different problem related to the system (\ref{af}), which will be
central point of our contribution.

\section{Degree of singularity. Problem setting}

\label{ds}

In what follows we consider our optimal control problem with {\it finite} or {\it infinite horizon}.

Due to lack of coercivity, "classical" minimizers for
(\ref{j0t})-(\ref{af})-(\ref{ep}) do not, in general, exist. It is
known that for generic boundary conditions, minimizing sequences of
classical controls usually converge to some 'generalized controls'
which may contain impulses or more complex singularities. For such
problems quasioptimal ($\varepsilon $-minimizing) controls
$u^{\varepsilon }$ are known to exhibit high-gain highly-oscillatory
behavior. It is expected $\lim\limits_{\varepsilon \rightarrow
0^{+}}\left\Vert u^{\varepsilon }\right\Vert _{L_{2}}=+\infty $ to
hold for any minimizing sequence $\left\{ u^{\varepsilon }\right\}
$. Still the asymptotics of growth of $\left\Vert u^{\varepsilon
}\right\Vert _{L_{2}}$ varies from problem to problem and therefore
this asymptotics can be used for measuring the degree of singularity
of the problem. This is also a problem of practical importance,
because suboptimal controls are harder to realize in practice when
"good" approximations of $\inf J_{0}^{T}$ require 'too high' gain
and 'too fast' oscillation.

In order to address this question, we introduce the following measure of
"singular behavior" of a problem (\ref{j0t})-(\ref{af})-(\ref{ep}).
\smallskip

\begin{definition}
\label{degsing} In the finite horizon case the degree of singularity
of the problem (\ref{j0t})-(\ref{af})-(\ref{ep}) is
\begin{equation*}
\sigma_T=\limsup \limits_{\varepsilon \rightarrow 0^+} \frac{ \inf
\left\{ \ln \left\|u\right\|_{L_2} : J_0^T(u) \leq \inf J_0^T +
\varepsilon , \ \left|x_u(T)-x_T \right| < \varepsilon \right\} } {
\ln \frac{1}{\varepsilon} } .
\end{equation*}
In the infinite horizon case the degree of singularity of the
problem (\ref{j0t})-(\ref{af}) is
\begin{equation*}
\sigma_\infty=\limsup \limits_{\varepsilon \rightarrow 0^+}\frac{
\inf \left\{\ln \left\|u\right\|_{L_2}: J_0^\infty(u) \leq \inf
J_0^\infty + \varepsilon \right\} }{\ln \frac{1}{\varepsilon}} . \
\Box
\end{equation*}
\end{definition}

Our main goal from now on will be computation of degree of
singularity for various optimal control-affine problems. In the next
Section we provide a complete analysis for singular linear quadratic
problems.

\section{Singular linear-quadratic case}

\label{S2}

In this section we discuss the relationship between the degree of
singularity $\sigma_T $ and the structure of generalized minimizers
in the singular linear-quadratic case. We believe this relation
provides a compelling evidence for the usefulness of degree of
singularity for measuring singular behavior of minimizing sequences.

In \cite{Guer}, a definition of order of singularity for LQ problem
has been introduced and it was shown that singular linear-quadratic
problems can be classified according to it. This order of
singularity is an integer $r\leq n$, $n$ being the dimension of the
state space. If $\inf \limits_{u}J_{0}^{T}(u)>-\infty $, then a
problem of order $r$ admits a generalized minimizer in the Sobolev
space $H_{-r}$.

We will show that the degree of singularity $\sigma_T $ from
Definition~\ref{degsing} is tightly related to the order of
singularity of a problem. For (singular) LQ problems with
state-quadratic integrand $x^{\prime }Px$ and $P>0$ it is shown that
$\sigma_T =\frac{1}{2}$, while order of singularity equals $1$. For
an LQ problem, with more general functional (\ref{j0tG}), order of
singularity $r$ and generic boundary data, $\sigma_T
=r-\frac{1}{2}$. When nongeneric boundary conditions are imposed,
one can show that $\sigma_T $ admits values from a finite set. These
values correspond to a stratification of the space of boundary data,
which is related to results in \cite{Jurdj}, \cite{Jurdj-Kupka},
\cite{Willems} and \cite{Guer}.

The content of Subsections~\ref{SS2a} to \ref{SS2d} below is
essentially a brief sketch of the results contained in \cite{Guer},
which are essential for the computation of $\sigma_T $.
Subsection~\ref{SS2e} contains an important technical result
(Proposition~\ref{P7}) regarding approximation of distributions from
Sobolev space $H_{-r}$. This result is applied in Subsection
\ref{SS2f} to computation of the values of degree of singularity $
\sigma_T $.

\subsection{Assumptions\label{SS2a}}

Along this Section the controlled dynamics (\ref{af}) is linear
time-invariant:
\begin{equation}
\dot{x}=Ax+Bu,\qquad x(0)=x_{0}.  \label{ls}
\end{equation}%
The end-point condition is (\ref{ep}). The cost functional, we consider,
will be more general than (\ref{j0t}):
\begin{equation}
J_{0}^{T}(u)=\int_{0}^{T}x_{u}^{\prime }Px_{u}+2u^{\prime }Qx_{u}+u^{\prime
}Ru\,d\tau ,  \label{j0tG}
\end{equation}%
where $R$ is a symmetric nonnegative matrix. If $R$ is positive
definite, then there exists analytic minimizing control for this
problem (at least for sufficiently small $T$) and hence the degree
of singularity $\sigma_T=0$. Therefore the case of interest is the
singular one, where $R$ possesses a nontrivial kernel.

We assume the following to hold.

\begin{assumption}
\label{nogauge} Let the matrices $A, \ B, \ P, \ Q, \ R$ in (\ref{ls}),(\ref%
{j0tG}) be such that for $x_{0} = x_T=0$ and each $T\in ] 0,+\infty [ $,
there exists a subspace $\mathcal{S}^{+}_T$, of finite codimension in $%
L_{2}^{k}[0,T]$, such that $J_{0}^{T}>0$ on $\mathcal{S}^{+}_T\setminus
\{0\} $ (the subspace $\mathcal{S}^{+}_T$ and its codimension may depend on $%
T$). $\square $
\end{assumption}

Assumption \ref{nogauge} may look not very natural, but as we will now
explain, it is closely related to finiteness of $\inf J_{0}^{T}$.

If there exist $T\in ]0,+\infty \lbrack $ and an infinite-dimensional
subspace, $\mathcal{S}^{-}\subset L_{2}^{k}[0,T]$ such that $J_{0}^{T}(u)<0$
in $\mathcal{S}^{-}\setminus \{0\}$, then one can prove that $%
\inf\limits_{u}J_{0}^{T}(u)=-\infty $ holds for every $T>0$ and any boundary
conditions. Thus finiteness of $\inf J_{0}^{T}$ requires that for each $T
\in ]0,+\infty [$ there exist some subspace of finite codimension in $%
L_{2}^{k}[0,T]$ where $J_{0}^{T}$ is non-negative. In this case, the
only way in which Assumption \ref{nogauge} can fail is when the
quadratic form $u \mapsto J_{0}^{T}(u)$ has infinite-dimensional
kernel. In \cite{Guer} it is shown how this kernel can be "factored
out". Naturally, Assumption~\ref{nogauge} will hold after such a
factorization.

Resuming, we may think of Assumption \ref{nogauge} as of a version
of the more intuitive

\begin{assumption}
\label{finite}$\inf J_{0}^{T}(u)>-\infty $ holds for each boundary data $%
(x_{0},x_T)$ (for each initial data $x_{0}$, when $T=+\infty $). \ $\square$
\end{assumption}

Assumptions \ref{nogauge} and \ref{finite} are closely related but
not equivalent. Using the first one is more convenient from the
technical viewpoint. A complete study of problems
(\ref{ls})-(\ref{ep})-(\ref{j0tG}) which satisfy
$\inf\limits_{u}J_{0}^{T}(u)>-\infty $ can be found in \cite{Guer}.
Therein it is shown how Assumption \ref{nogauge} can be checked
using only linear algebra computations.

Note that, in the \textit{finite-horizon} case neither $P$ nor the quadratic
form $(x,u)\mapsto x^{\prime }Px+2u^{\prime }Qx+u^{\prime }Ru$ need to be
nonnegative for $\inf\limits_{u}J_{0}^{T}(u)>- \infty $ to hold. In the
infinite-horizon case we will require this latter nonnegativity.

\subsection{Desingularization of LQ problems\label{SS2b}}

Provided Assumption~\ref{nogauge} holds, a singular linear-quadratic
problem (\ref{ls})-(\ref{ep})-(\ref{j0tG}) can be reduced to a
regular problem, i.e. to an LQ problem with quadratic cost, which is
strictly convex with respect to control. This is done by the
following multistep procedure (for a detailed account of a more
general procedure without Assumption \ref{nogauge}, see
\cite{Guer}).

Let $\phi :L_{2,loc}\mapsto L_{2,loc}$ be the primitivization:
\[
\phi u(t)=\int_{0}^{t}u(\tau )\,d\tau ,\quad \forall u\in L_{2,loc}.
\]

By choosing a suitable coordinate system in the space of control variables,
we may assume without loss of generality that the nonnegative matrix $R$ in (%
\ref{j0tG}) is of the form $R=\left(
\begin{array}{cc}
R_{0,0} & 0 \\
0 & 0%
\end{array}%
\right) ,$ where $R_{0,0}\in \mathbb{R}^{k_{0}\times k_{0}},\ R_{0,0}>0$. We
consider the corresponding splitting of the vectors $u=\left(
u_{0},u_{1}\right) \in \mathbb{R}^{k},\ u_{0}\in \mathbb{R}^{k_{0}}$ and of
the matrices
\begin{equation*}
B=\left( B_{0,0},B_{0,1}\right) \in \mathbb{R}^{n\times k}, \ B_{0,0}\in
\mathbb{R}^{n\times k_{0}}, \ Q=\left(
\begin{array}{c}
\!Q_{0,0}\! \\
\!Q_{0,1}\!%
\end{array}%
\right) \in \mathbb{R}^{k\times n}, \ Q_{0,0}\in \mathbb{R}^{k_{0}\times n}.
\end{equation*}

Let us introduce the operator $\gamma : L_{2,loc}^k \mapsto
L_{2,loc}^k$
\begin{equation*}
\gamma u = \left( u_0 + R_{0,0} ^{-1}
\left(Q_{0,0}B_{0,1}-B_{0,0}^{\prime}Q_{0,1}^{\prime}\right) \phi u_1, \
\phi u_1 \right).
\end{equation*}
The following Proposition represents the trajectory $x_{x_{0},\gamma
u}$ and the value of the functional $J_0(u)$ via solution of an LQ
problem, which is 'less singular'. Due to it the representation
below is called \textit{desingularization procedure}.

\begin{proposition}[\protect\cite{Guer}]
\label{P6}For every $x_{0}\in \mathbb{R}^{n}$, $u\in L_{2}^{k}\left[ 0,T%
\right] $, there holds:%
\begin{align*}
x_{x_{0},u}\left( t\right) = & x_{x_{0},\gamma u}^{1}\left( t\right)
+B_{0,1}\phi u_{1}\left( t\right) ,
\\
J_{x_{0}}\left( u\right) =& \int_{0}^{T} x_{x_{0},\gamma u}^{1\prime
}Px_{x_{0},\gamma u}^{1} + 2\left( \gamma u\right) ^{\prime
}Q_{1}x_{x_{0},\gamma u}^{1} + \left( \gamma u\right)^{\prime
}R_{1}\gamma u~d\tau+ \\
& +\int_{0}^{T}u_{1}^{\prime }\left( Q_{0,1}B_{0,1}-B_{0,1}^{\prime
}Q_{0,1}^{\prime }\right) \phi u_{1}~d\tau + \\
& +2\phi u_{1}\left( T\right) ^{\prime }Q_{0,1}x_{x_{0},\gamma
u}^{1}\left( T\right) + \phi u_{1}\left( T\right) ^{\prime
}Q_{0,1}B_{0,1}\phi u_{1}\left( T\right) ,
\end{align*}
where $x_{x_{0},v}^{1}$ denotes the trajectory of the system
\begin{equation*}
\dot{x}=Ax+B_{1}v,\qquad x\left( 0\right) =x_{0},
\end{equation*}%
\begin{eqnarray*}
&& \hspace{-0.5cm} B_{1} = \left( B_{0,0},\,B_{1,1}\right) ; \\
&& B_{1,1}=
\begin{array}[t]{l}
\left( A-B_{0,0}R_{0,0}^{-1}Q_{0,0}\right) B_{0,1}+ \smallskip
B_{0,0}R_{0,0}^{-1}B_{0,0}^{\prime }Q_{0,1}^{\prime };%
\end{array}
\\
&& \hspace{-0.5cm} Q_{1} = \left(
\begin{array}{c}
Q_{0,0} \\
Q_{1,1}%
\end{array}%
\right) ; \\
&& Q_{1,1}=
\begin{array}[t]{l}
B_{0,1}^{\prime }\left( P-Q_{0,0}R_{0,0}^{-1}Q_{0,0}\right) - \smallskip
Q_{0,1}\left( A-B_{0,0}R_{0,0}^{-1}Q_{0,0}\right) ;%
\end{array}
\\
&& \hspace{-0.5cm} R_{1} = \left(
\begin{array}{cc}
R_{0,0} & 0\medskip \\
0 & \tilde{R}_{1}%
\end{array}%
\right) ; \qquad \tilde{R}_{1}=Q_{1,1}B_{0,1}-B_{1,1}^{\prime
}Q_{0,1}^{\prime }.
\end{eqnarray*}%
For Assumption~\ref{nogauge} to hold, there must be
\begin{equation}  \label{Z13}
Q_{0,1}B_{0,1}-B_{0,1}^{\prime }Q_{0,1}^{\prime }=0
\end{equation}
and $R_{1}\geq 0$.$\square $
\end{proposition}

If Assumption \ref{nogauge} holds for the $5$-ple $\left(
A,B,P,Q,R\right) $, then it also will hold for the $5$-ple $\left(
A,B_{1},P,Q_{1},R_{1}\right) $, which corresponds to the
desingularized problem. The quadratic form $\left( x,u\right)
\mapsto x^{\prime }Px+2u^{\prime }Qx+u^{\prime }Ru$ is nonnegative
if and only if the quadratic form $\left( x,u\right) \mapsto
x^{\prime }Px+2u^{\prime }Q_{1}x+u^{\prime }R_{1}u$ is. If $R_{1}$
has a nontrivial kernel, then we can repeat the procedure obtaining
a sequence $\left( A,B_{i},P,Q_{i},R_{i}\right) $, $i=1,2,...$. The
following Proposition states that this sequence must be finite.

\begin{proposition}
\label{P8} If Assumption~\ref{nogauge} holds, then there exists an
integer $ r\leq n$ such that $R_{r}>0$.$\square $
\end{proposition}

In \cite{Guer} the integer $r$ is called \textbf{order of singularity} of
the problem (\ref{ls})-(\ref{ep})-(\ref{j0tG}).

Without loss of generality, we may assume that the coordinates of the space
of control variables are such that the matrices $\left(
B_{i},Q_{i},R_{i}\right) $, obtained at each desingularization step, have
block structure
\begin{eqnarray*}
&& R_{i} = \mbox{diag}\left( R_{0,0},R_{1,1},\ldots ,R_{i,i},0\right) , \\
&& B_{i} = \left( B_{0,0}\ B_{1,1}\ \cdots \ B_{i,i}\ B_{i,i+1}\ \cdots \
B_{i,r}\right) , \\
&& Q_{i}^{\prime } = \left( Q_{0,0}^{\prime }\ Q_{1,1}^{\prime }\ \cdots \
Q_{i,i}^{\prime }\ Q_{i,i+1}^{\prime }\ \cdots \ Q_{i,r}^{\prime }\right) .
\end{eqnarray*}%
Let us introduce operator $\gamma _{r}=\left( \gamma _{r,0},\gamma
_{r,1},...\gamma _{r,r}\right) :L_{2,loc}^{k}\mapsto L_{2,loc}^{k}$ as
follows
\begin{eqnarray}
&& \gamma _{r,i}u = \begin{array}[t]{l} \phi ^{i}u_{i}+
R_{i,i}^{-1}\sum\limits_{i\leq j<l\leq r}\left(
Q_{i,i}B_{j,l}-B_{i,i}^{\prime }Q_{j,l}^{\prime }\right) \phi
^{j+1}u_{l}, \\
\hspace{6cm} i=0,1,...,(r-1);
\end{array}  \notag \\
&& \gamma _{r,r}u = \phi ^{r}u_{r}.  \label{gr}
\end{eqnarray}%
This operator is used in the next Subsection to introduce a suitable
topology in the space of generalized controls.

Applying Proposition~\ref{P6} consequently $r$ times, we arrive to
the following corollary.

\begin{proposition}
\label{despr} The trajectory $x_{x_{0},u}$ can be represented as
\begin{equation}
x_{x_{0},u}=x_{x_{0},\gamma _{r}u}^{r}+\sum\limits_{0\leq i<j\leq
r}B_{i,j}\phi ^{i+1}u_{j},  \label{xred}
\end{equation}%
where $x_{x_{0},\gamma _{r}u}^{r}$ denotes the trajectory of the system
\begin{equation}
\dot{x}=Ax+B_{r}\gamma _{r}u,\qquad x(0)=x_{0}.  \label{eqgr}
\end{equation}

If Assumption~\ref{nogauge} holds and $T<+\infty $, then the
functional $ J_{0}^{T}$ can be represented as
\begin{eqnarray*}
J_0^T (u) &=& \int _0^T {x^r_{\gamma_r u}} ^{\prime}P x^r_{\gamma_r
u} + 2 \gamma_r u ^{\prime}Q_r x^r_{\gamma_r u} + \gamma_r u
^{\prime}R_r \gamma_r
u \, dt + \\
&& + 2 \sum \limits_{0 \leq i < j \leq r}
\phi^{i+1}u_j(T)^{\prime}Q_{i,j}
x^r_{\gamma_r u} (T) + \\
&& +\left( \sum \limits_{0 \leq i < j \leq r}
\phi^{i+1}u_j(T)^{\prime}Q_{i,j} \right) \cdot \left( \sum
\limits_{0 \leq i < j \leq r} B_{i,j} \phi^{i+1}u_j(T)\right),
\end{eqnarray*}
where $R_r>0. \ \Box$
\end{proposition}

It turns out that the nonintegral terms in the latter representation
of $ J_0^T$ depend only on $T$, $x_0$ and $x_T$, and hence
\begin{eqnarray}\label{J0Tred}
J_0^T (u) &=&   \int _0^T {x^r_{\gamma_r u}} ^{\prime}P
x^r_{\gamma_r u} +2 \gamma_r u ^{\prime}Q_r x^r_{\gamma_r u} +
\gamma_r u ^{\prime}R_r \gamma_r u \, dt + \\
&& +C_r^T \left( x_0, x_T \right), \nonumber
\end{eqnarray}
where $C_r^T \left( x_0, x_T \right) $ is quadratic with respect to
$\left( x_0, x_T \right)$.

\begin{remark}
Recall that in the infinite-time horizon version of the problem we
assume nonnegativeness of the quadratic form $\left( x,u\right)
\mapsto x^{\prime }Px+2u^{\prime }Qx+u^{\prime }Ru$. In this case
the nonintegral terms vanish and the functional takes form \[
J_{0}^{\infty }(u)= \int_{0}^{\infty }{x_{\gamma _{r}u}^{r}}^{\prime
}Px_{\gamma _{r}u}^{r}+2\gamma _{r}u^{\prime }Q_{r}x_{\gamma
_{r}u}^{r}+ \gamma
_{r}u^{\prime }R_{r}\gamma _{r}u\,dt.%
\]
\end{remark}

\subsection{Weak norms, generalized controls and trajectories\label{SS2c}}

Consider the following norms in the space of inputs $L_{2}^{k}[0,T]$ and in
the space of trajectories $L_{2}^{n}[0,T]$ (we embed absolutely continuous
trajectories in $L_{2}^{n}$):
\begin{eqnarray*}
&& \left\Vert u\right\Vert _{\gamma _{r}[0,T]} = \left\Vert \gamma
_{r}u\right\Vert _{L_{2}^{k}[0,T]}, u\in L_{2}^{k}[0,T]; \\
&& \left\Vert u\right\Vert _{\overline{\gamma }_{r}[0,T]} = \hspace{-0.2cm}
\begin{array}[t]{l}
\left\Vert \gamma _{r}u\right\Vert _{L_{2}^{k}[0,T]}+\sum\limits_{1\leq
i\leq j\leq r}\left\vert \phi ^{i}u_{j}(T)\right\vert ,\ u\in L_{2}^{k}[0,T];%
\end{array}
\\
&& \left\Vert x\right\Vert _{H_{-r}^{n}[0,T]} = \left\Vert \phi
^{r}x\right\Vert _{L_{2}^{n}[0,T]},\qquad x\in L_{2}^{n}[0,T].
\end{eqnarray*}%
Fix $T\in ]0,+\infty \lbrack $ and let

\begin{itemize}
\item $\mathcal{U}_{\gamma _{r}[0,T]}$ be the topological completion of $%
L_{2}^{k}[0,T]$ with respect to $\left\Vert \cdot \right\Vert _{\gamma
_{r}[0,T]}$; \smallskip

\item $\mathcal{U}_{\overline{\gamma }_{r}[0,T]}$ be the topological
completion of $L_{2}^{k}[0,T]$ with respect to $\left\Vert \cdot \right\Vert
_{\overline{\gamma }_{r}[0,T]}$; \smallskip

\item $H^n_{-r}[ 0,T ]$ be the topological completion of $L_{2}^{n}[0,T]$
with respect to $\left\Vert \cdot \right\Vert _{H_{-r}^{n}[0,T]}$.
\end{itemize}

\smallskip

The following holds true (see \cite{Guer})

\begin{proposition}
\label{cont} The input-to-trajectory map $u\mapsto x_{x_{0},u}$, is
uniformly continuous with respect to the norm $\left\Vert \cdot \right\Vert
_{\gamma _{r}[0,T]}$ of inputs and the norm $\left\Vert \cdot \right\Vert
_{H_{-(r-1)}^{n}[0,T]}$ of trajectories. The functional $J_{0}^{T}(u)$ is
locally uniformly continuous in the norm $\left\Vert \cdot \right\Vert _{%
\overline{\gamma }_{r}[0,T]}$ of inputs. $\square $
\end{proposition}

As a corollary we obtain.

\begin{proposition}
The input-to-trajectory map $u\mapsto x_{x_{0},u}$ admits a unique
continuous extension with domain $\mathcal{U}_{\gamma _{r}[0,T]}$
and range in $H_{-(r-1)}^{n}[0,T]$, while the functional
$J_{0}^{T}(u)$ admits a unique continuous extension onto
$\mathcal{U}_{\overline{\gamma }_{r}[0,T]}$. These extensions can be
defined by equalities (\ref{J0Tred}) and (\ref{xred}), respectively.
$\square $
\end{proposition}

We denote $\mathcal{U}_{\gamma _{r}[0,+\infty [ }$ the topological
completion of $L_{2,loc}^{k}$ with respect to convergence in all the
norms $ \left\| \cdot \right\| _{\gamma _{r}[0,T]}$, (with $r$ fixed
like in Proposition \ref{P8} and $T$ ranging in $]0, +\infty[$).
Similarly, $H_{-r}^{n}[0,+\infty [ $ is the topological completion
of $L_{2,loc}^{n}$ with respect to convergence in all the norms $
\left\| \cdot \right\| _{H_{-r}^{n}[0,T]}$ ($T < +\infty$, $r$
fixed).

\begin{proposition}
\label{extinf} The map $u\mapsto x_{x_{0},u}$ admits a unique continuous
extension onto $\mathcal{U}_{\gamma _{r}[0,+\infty \lbrack }$. If the
quadratic form $(x,u)\mapsto x^{\prime }Px+2u^{\prime }Qx+u^{\prime }Ru$ is
nonnegative, then the map $u\mapsto J_{0}^{+\infty }(u)$ admits a unique
extension onto $\mathcal{U}_{\gamma _{r}[0,+\infty \lbrack }$. $\square$
\end{proposition}

Note that any function $v\in L_{2,loc}^{k}$ defines uniquely a
generalized control $u\in \mathcal{U}_{\overline{\gamma }_{r}[0,T]}$
such that $\gamma _{r}u=v$ and $\phi ^i u_j (T) = V_{i,j}$, $1 \leq
i \leq j \leq r$ ($V_{i,j}$ fixed constant vectors of suitable
dimensions); the metric in the space $u\in \mathcal{U}_{
\overline{\gamma }_{r}[0,T]}$ is induced by $L_{2}$-metric in the
space of $ v= \gamma _{r}u$.

\subsection{Desingularized LQ problems and generalized solutions\label{SS2d}}

According to Proposition \ref{despr} the problem
(\ref{ls})-(\ref{ep})-(\ref{j0tG}) can be transformed into the
following regular LQ problem with the control $v=\gamma _{r}u$:
\begin{eqnarray}
&& \hspace{-0.5cm} J_{red}^{T}(v) = \int_{0}^{T}x^{\prime
}Px+2v^{\prime }Q_{r}x+v^{\prime }R_{r}v\,dt
\smallskip \rightarrow \min,  \label{Jred} \\
&& \hspace{-0.5cm} \dot{x} = Ax+B_{r}v, \quad v\in
L_{2}^{k}[0,T],\qquad
x(0)=x_{0},  \label{xr} \\
&& \hspace{-0.5cm} x(T) \in x_{T}+\mathop{\mbox{span}}\nolimits\left\{
B_{i,j},\,0\leq i<j\leq r\right\} ,  \label{epred}
\end{eqnarray}%
with the endpoint condition (\ref{epred}) being dropped in case
$T=+\infty $. Since $R_r>0$, classical existence theory applies to
(\ref{Jred})-(\ref{xr})-(\ref{epred}).

\begin{definition}
\label{D coerc} A functional $\theta$ defined in a normed space
$(X,\| \cdot \| )$ is said to be quadratically coercive if there are
constants $a \in \mathbb R$, $b>0$ such that
\[
\theta( \xi ) \geq a +b \|\xi \|^2, \qquad \forall \xi \in X. \
\square
\]
\end{definition}

Due to the relationship $v=\gamma _{r}u$, between the new and the
original controls, the following results hold true for
(\ref{ls})-(\ref{ep})-(\ref{j0tG}).
\smallskip

\begin{proposition}
\label{thcoercive} Let Assumption \ref{nogauge} hold. Then:

\begin{itemize}
\item[i)] for each sufficiently small $T>0$, the functional $J_{0}^{T}(u)$
is quadratically coercive on $\left\{ u\in \mathcal{U}_{\overline{\gamma }
_{r}[0,T]}:\phi ^{i}u_{j}(T)=0, \ 1\leq i\leq j\leq r\right\} $; \smallskip

\item[ii)] for any $T>0$ the functional $J_{0}^{T}(u)$ is quadratically
coercive on a subspace of finite codimension in
$\mathcal{U}_{\overline{ \gamma } _{r}[0,T]}$; \smallskip

\item[iii)] If $T=+\infty $, and the quadratic form $u\mapsto x^{\prime
}Px+2u^{\prime }Qx+u^{\prime }Ru$ is nonnegative, then the
functional $ J_{0}^{\infty }(u)$ is quadratically coercive on
$\mathcal{U}_{\gamma _{r}[0,+\infty \lbrack }.\ \square $
\end{itemize}

\end{proposition}

Using Proposition \ref{thcoercive}, classical existence theory and the
relationship $v= \gamma _r u$ we obtain the following:

\begin{proposition}
\label{CExUni} Let Assumption \ref{nogauge} hold. \medskip
\newline
For the finite-horizon case: there exists $T_0 \in ] 0,+ \infty ] $
such that

\begin{itemize}
\item[i)] For each $T \in ] 0, T_0 [$, and every $x_{T}$ accessible from $%
x_{0}$, problem (\ref{ls})-(\ref{ep})-(\ref{j0tG}) admits a unique
generalized solution, $\left( \hat{ u}, \hat{x}\right) \in \mathcal{U}_{%
\overline{\gamma }_{r}[0,T]}\times H^n_{-(r-1)}[0,T]$; \smallskip

\item[ii)] For any $T>T_0$, and every $x_T$ accessible from $x_0$, $%
\inf\limits_u J^T_0 = - \infty $ holds;
\end{itemize}

\noindent For the infinite-horizon case:

\begin{itemize}
\item[iii)] If the quadratic form $(x,u) \mapsto x^{\prime}Px+2u^{\prime}Qx
+u^{\prime}Ru$ is nonnegative and system (\ref{ls}) is feedback
stabilizable, then problem (\ref{ls})-(\ref{j0tG}) admits a unique
generalized solution, $\left( \hat u, \hat x \right) \in
\mathcal{U}_{\gamma _r [0,+ \infty[} \times H_{-(r-1)}[0,+ \infty[.
\ \square$
\end{itemize}

\end{proposition}

For the proofs of these results, see \cite{Guer}. We will now briefly
discuss the structure of generalized optimal solutions.

For regular linear-quadratic problems any optimal control
$v^*(\cdot)$ satisfies the Pontryagin maximum principle and is
analytic. %
From the desingularization procedure (Proposition~\ref{despr}) there
follows that the corresponding optimal generalized control for the
original singular LQ problem (\ref{ls})-(\ref{ep})-(\ref{j0tG})
satisfies the relationship $\gamma_r u ^* = v^*$, with $\gamma_r$
defined by (\ref{gr}). Hence it is a sum of an analytic function and
a distribution of order $\leq r$. This distribution is supported at
the points $t=0$ and $t=T$.

The corresponding generalized optimal trajectory is analytic on $]0,T[$ and
may happen to be discontinuous at points $t=0$ and $t=T$. The generalized
trajectory "jumps" at $t=0$ from $x(0)=x_0$ to the point
\begin{equation*}
x(0^+)=x_0+ \sum \limits _{0\leq i < j \leq r} B_{i,j}\phi^{i+1}u_j(0^+).
\end{equation*}
For $t \in ]0, T[$ it coincides with the analytic curve
\begin{equation*}
x(t)=x^r_{x_0,\gamma_r u}(t)+ \sum\limits_{0 \leq i<j\leq r}B_{i,j} \phi
^{i+1}u_j (t),
\end{equation*}
where $x^r_{x_0,\gamma_r u}$ is the trajectory of (\ref{eqgr}) driven by the
analytic control $\gamma_ru$. The generalized trajectory terminates with a
"jump" from the point
\begin{equation*}
x(T^-)=x^r_{\gamma_r u}(T)+ \sum\limits_{0\leq i < j \leq r}B_{i,j}\phi
^{i+1}u_j(T^-)
\end{equation*}
to the point
\begin{equation*}
x(T)=x^r_{\gamma_r u}(T)+ \sum\limits_{0\leq i < j \leq r}B_{i,j}\phi
^{i+1}u_j(T)=x_T
\end{equation*}
(the vectors $\phi^i u (0^+)$ and $\phi^i u (T^-)$ are well defined for all $%
i \geq 0$).

Assumption \ref{nogauge} guarantees that for any jump $
x(t_{0}^{+})-x(t_{0}^{-})$ belonging to $\mathop{\mbox{span}}
\nolimits\left\{ B_{i,j},\,0\leq i<j\leq r\right\}$ there exists a
unique distribution $\Delta \in \mathcal{U}_{\gamma _{r}[0,+\infty
\lbrack }$ supported at $\{t_{0}\}$, such that
\begin{equation*}
x(t_{0}^{+})-x(t_{0}^{-})=\sum\limits_{0\leq i<j\leq r}B_{i,j}\phi
^{i+1}\Delta _{j}.
\end{equation*}
The following result states an important property of the optimal
synthesis for problem (\ref{ls})-(\ref{j0tG})-(\ref{ep}) (see
\cite{Jurdj}, \cite{Jurdj-Kupka} and a result in \cite[Ch.6]{Or1},
which claims a minimizing control to be "generically" a sum of a
continuous control with impulses of different orders located at the
initial and the final point of the time interval).
\smallskip

\begin{proposition}
\label{thJurdj} Let Assumption~\ref{nogauge} hold and the infimum of
the problem (\ref{ls})-(\ref{ep})-(\ref{j0tG}) be finite. Consider
the subspace $ \mathcal{B}_r=\mathop{\mbox{span}}\nolimits\left\{
B_{i,j},\,0\leq i<j\leq r\right\} $.
\newline
For all
\begin{equation}
\tilde{x}_{0} \in x_{0}+\mathcal{B}_r, \qquad \tilde{x}_{T} \in
x_{T}+\mathcal{B} _r,   \label{bdxT}
\end{equation}%
the problem with boundary conditions $x(0)=\tilde{x}_{0}$, $x(T)=\tilde{x}
_{T}$ admits a generalized optimal solution (control and trajectory). For
each boundary data from the sets (\ref{bdxT}) there exists a generalized
optimal solution coinciding in the interval $]0,T[$ with the analytic arc of
the solution for the data $(x_0,x_T) . \ \square $
\end{proposition}

Suppose Assumption \ref{nogauge} to hold. Let $m=rank\left(
B,~AB,...,A^{n-1}B\right) $, and let $p=\dim \mathop{\mbox{span}}%
\nolimits\left\{ B_{i,j},\ 0\leq i<j\leq r\right\} $. Fix $T\in \left]
0,+\infty \right[ $ and let $\mathcal{X}\subset \mathbb{R} ^{2n}$, denote
the set of pairs $\left( x_{0},x_{T}\right) $ for which the problem (\ref{ls}%
)-(\ref{j0tG})-(\ref{ep}) possesses classical optimal solution, $\left(
u,x_{u}\right) \in L_{2}^{k}[0,T]\times AC^{n}[0,T]$. Propositions \ref%
{CExUni} and \ref{thJurdj} imply that, if $T>0$ is sufficiently small, then $%
\mathcal{X}$ is a linear subspace of dimension $n+m-2p$. Further, existence
and uniqueness of generalized optimal solutions implies that any pair $%
\left( x_{0},x_{T}\right) \in \mathbb{R}^{2n}$, such that $x_{T}$ is
reachable from $x_{0}$, admits a unique decomposition
\begin{equation}
\left( x_{0},x_{T}\right) =
\begin{array}[t]{l}
\left( \tilde{x}_{0},\tilde{x}_{T}\right) + \sum\limits_{0\leq i<j\leq
r}\left(
\begin{array}{c}
B_{i,j} \\
0%
\end{array}%
\right) \alpha _{i,j}+ \sum\limits_{0\leq i<j\leq r}\left(
\begin{array}{c}
0 \\
B_{i,j}%
\end{array}%
\right) \beta _{i,j},%
\end{array}
\label{xb}
\end{equation}%
with $\left( \tilde{x}_{0},\tilde{x}_{T}\right) \in \mathcal{X}$, and $%
\alpha _{i,j},\ \beta _{i,j}$ being vectors of appropriate dimensions. The
discontinuities of the generalized optimal trajectories are computed as
\begin{eqnarray*}
&& x(0^{+})-x_{0} = \sum\limits_{0\leq i<j\leq r}B_{i,j}\alpha _{i,j},
\\
&& x_{T}-x(T^{-}) = \sum\limits_{0\leq i<j\leq r}B_{i,j}\beta _{i,j}.
\end{eqnarray*}

\subsection{Approximation of distributions\label{SS2e}}

The following Proposition gives the asymptotics of approximations of
some important distributions by square-integrable functions.

\begin{proposition}
\label{P7}Consider a $m^{th}$-order distribution of the form
\[
v=\delta ^{\left( m-1\right) }+\sum_{i=0}^{m-2}\alpha _{i}\delta
^{\left( i\right) },
\]
where $\alpha _{i}\in \mathbb{R}$, $i=0,1,...,m-2$, and $\delta
^{\left( i\right) }$ denotes the $i^{th}$ generalized derivative of
Dirac's "delta function".

For every integer $p\geq m$ there holds
\begin{equation*}
\begin{array}[t]{l}
\lim\limits_{\eta \rightarrow 0^{+}}\frac{\inf \left\{ \log \left\Vert
u\right\Vert _{L_{2}\left[ 0,T\right] }:\left\Vert u-v\right\Vert _{H_{-p}%
\left[ 0,T\right] }<\eta \right\} }{\log \frac{1}{\eta }}=\frac{2m-1}{%
2(p-m)+1}.\square%
\end{array}%
\end{equation*}
\end{proposition}

The rather technical proof of Proposition~\ref{P7} can be found in
Subsection~\ref{SS4a}.

\subsection{Degree of singularity of LQ problems\label{SS2f}}

For the finite horizon case ($T<+\infty $) the following result holds

\begin{theorem}
\label{c muT} Consider the problem (\ref{ls})-(\ref{ep})-(\ref{j0tG}).
Suppose Assumption \ref{nogauge} holds and let $\left( x_{0},x_{T}\right) $,
$\left( \tilde{x}_{0},\tilde{x}_{T}\right) $, $\alpha _{i,j}$, $\beta _{i,j}$
be as in (\ref{xb}).

\begin{enumerate}
\item If $\alpha_{i,j}=0, \ \beta_{i,j}=0, \ 0 \leq i < j \leq r$, and the
optimal control is not identically zero, then $\sigma_T =0$; if the
optimal control is zero then $ \sigma_T =-\infty$; \smallskip

\item If either $\alpha _{i,j}\neq 0$ or $\beta _{i,j}\neq 0$ for some $%
(i,j) $, then
\begin{equation}  \label{mulist}
\sigma _{T}=\max \limits_{%
\begin{array}{c}
0\leq i<j\leq r, \\
(\alpha _{i,j}\neq 0\ \mathrm{or}\ \beta _{i,j}\neq 0)%
\end{array}%
} \frac{i+1/2}{2(j-i)-1} .\ \square
\end{equation}
\end{enumerate}
\end{theorem}

Theorem \ref{c muT}, together with the decomposition (\ref{xb}) results in
the following description of the geometry of singularity of LQ problems.

\begin{theorem}
\label{strat} The space $\mathbb{R} ^{n}\times \mathbb{R}^{n}$ of boundary
data admits a stratification into linear subspaces. The directing linear
subspaces of strata are spanned by the columns of the matrices
\begin{equation*}
\left(
\begin{array}{cc}
B_{i,j} & 0 \\
0 & B_{i,j}%
\end{array}%
\right) , \qquad 0\leq i<j\leq r ,
\end{equation*}
with $B_{i,j}$ defined in Subsection~\ref{SS2b}. Different strata
correspond to different distributional components of optimal
generalized controls. For generic boundary data (the largest
stratum) optimal generalized controls contain distributional
components of the form $\delta ^{(r-1)}$; the degree of singularity
equals $r-1/2$. When passing to the strata of smaller dimensions the
order of singularity $\sigma_T $ decreases, admitting values from
the list
\begin{equation}
\left\{ \frac{i+1/2}{2(j-i)-1}:\ 0\leq i<j\leq r,\right\} . \ \Box
\nonumber 
\end{equation}
\end{theorem}

The infinite horizon case ($T=+\infty $) is analogous to the finite horizon
case; one just has to deal with stratification of the space of initial data.

Let $\mathcal{X}_{0}\subset \mathbb{R} ^{n}$, denote the set of initial
points $x_{0}$ for which problem (\ref{ls})-(\ref{j0tG}) has a classical
optimal solution, $\left( u,x_{u}\right) \in L_{2}^{k}[0,+\infty \lbrack
\times AC^{n}[0,+\infty \lbrack $. If the quadratic form $(x,u) \mapsto
x^{\prime}Px+2u^{\prime}Qx+u^{\prime}Ru$ is nonnegative and system (\ref{ls}%
) is feedback stabilizable, then Theorem \ref{thJurdj} guarantees that $%
\mathcal{X}_{0}$ is a linear subspace of dimension $n-p$. Further, existence
and uniqueness of generalized optimal solutions implies that any initial
point $x_{0}\in \mathbb{R}^{n}$ admits a unique decomposition
\begin{equation*}
x_{0}=\tilde{x}_{0}+\sum\limits_{0\leq i<j\leq r}B_{i,j}\alpha _{i,j},
\end{equation*}%
with $\tilde{x}_{0}\in \mathcal{X}_{0}$, and $\alpha _{i,j}$ being vectors
of appropriate dimensions. The discontinuities of the generalized optimal
trajectories are
\begin{equation*}
x(0^{+})-x_{0}=\sum\limits_{0\leq i<j\leq r}B_{i,j}\alpha _{i,j}.
\end{equation*}%
Thus we have for the infinite horizon case the following analogous of
Theorem \ref{c muT}: \smallskip

\begin{theorem}
\label{c muInf} Consider the problem (\ref{ls})-(\ref{j0tG}), with $
T=+\infty $. Let Assumption \ref{nogauge} hold, the quadratic form
$(x,u) \mapsto x^{\prime}Px+2u^{\prime}Qx+u^{\prime}Ru$ be
nonnegative and system (\ref{ls}) be stabilizable by linear
feedback. Then

\begin{enumerate}
\item If $\alpha_{i,j}=0, \ 0 \leq i < j \leq r$, and the optimal control is
nonzero, then $\sigma_\infty =0$; $\sigma_\infty =-\infty$, if the
optimal control is zero; \smallskip

\item If $\alpha _{i,j}\neq 0$ for some $(i,j)$, then
\begin{equation*}
\sigma_\infty =\max\limits_{%
\begin{array}{c}
0\leq i<j\leq r \\
\alpha _{i,j}\neq 0%
\end{array}
} \frac{i+1/2}{2(j-i)-1} . \ \square
\end{equation*}
\end{enumerate}
\end{theorem}

\begin{remark}
Since the system (\ref{ls}) is linear, it follows that it is
stabilizable if and only if it is stabilizable by a linear feedback.
Therefore it is stabilizable if and only if it is stabilizable of
order $1$ in the sense of Definition \ref{detc}.
\end{remark}

\noindent \textbf{Proof of Theorem~\ref{c muT}.} Notice that $\alpha
_{i,j}=0,\ \beta _{i,j}=0$ for all $(i,j)$ such that $0\leq i<j\leq
r$ if and only if the optimal control is an analytic function in
$[0,T]$. Thus, assertion (1) follows immediately.

Otherwise the optimal control is the sum of an analytic function
with a distribution concentrated at points $0, T$. We will first
prove that $\sigma _{T}$ can not exceed the value (\ref{mulist}).

For each $j\in \left\{ 1,...,r\right\} $ let $p_{j}=\max \left\{ i:\alpha
_{i,j}\neq 0\text{ or }\beta _{i,j}\neq 0\right\} $. Let $\hat{u}=\left(
\hat{u}_{0},\hat{u}_{1},...,\hat{u}_{r}\right) $ denote the generalized
optimal control. Since (\ref{xb}) holds, then by equality (\ref{xred}) $%
u_{j} $ contains a distributional component of order $p_{j}+1$, supported at
$\left\{ 0,T\right\} $. Proposition \ref{P7} guarantees the existence of $%
\left\{ u_{r,\eta },~\eta >0\right\} $ such that
\begin{equation*}
\left\Vert \phi ^{r}\left( u_{r,\eta }-\hat{u}_{r}\right) \right\Vert _{L_{2}%
\left[ 0,T\right] }=O\left( \eta \right) , \ \left\Vert u_{r,\eta
}\right\Vert _{L_{2}\left[ 0,T\right] }=O\left( \eta^{-\frac{2p_{r}+1}{%
2\left( r-p_{r}\right) -1}}\right) .
\end{equation*}
Suppose that for some $j\geq 1$ we can chose $\left\{ \left( u_{j+1,\eta
},u_{j+2,\eta },u_{r,\eta }\right) ,~\eta >0\right\} $ such that
\begin{eqnarray*}
\sum_{q=j+1}^{r}\left\Vert \gamma _{r,q}\left( u_{\eta }-\hat{u}\right)
\right\Vert _{L_{2}\left[ 0,T\right] }=O\left( \eta \right) , \\
\sum_{q=j+1}^{r}\left\Vert u_{q,\eta }\right\Vert _{L_{2}\left[ 0,T\right]
}= O\left( \eta^{-\max_{q>j}\frac{2p_{q}+1}{2\left( q-p_{q}\right) -1}%
}\right).
\end{eqnarray*}

Since
\begin{equation*}
\gamma _{r,j}\hat{u}=\phi ^{j} \hat{u}_{j}+ R_{j,j}^{-1}\sum_{j\leq i<l\leq
r}\left( Q_{j,j}B_{i,l}-B_{j,j}^{\prime }Q_{i,l}^{\prime }\right) \phi ^{i+1}%
\hat{u}_{l}
\end{equation*}
are square-integrable and all distributional components of $\hat{u}$ are
supported at $\left\{ 0,T\right\} $, then for some constants $V_{i}^{0}$, $%
V_{i}^{T}$, $i=1,2,...,p_{j}$, and some square-integrable function, $w$
there holds
\begin{eqnarray*}
\hspace{0.3cm} \hat{u}_{j}\left( t\right) +
R_{j,j}^{-1}\sum\limits_{j\leq i<l\leq r}\left(
Q_{j,j}B_{i,l}-B_{j,j}^{\prime }Q_{i,l}^{\prime }\right)
\phi ^{i-j+1}\hat{u}_{l}\left( t\right) = \smallskip \\
= w\left( t\right) +\sum\limits_{i=1}^{p_{j}}\left( V_{i}^{0}\delta ^{\left(
i-1\right) }\left( t\right) +V_{i}^{T}\delta ^{\left( i-1\right) }\left(
t-T\right) \right) .
\end{eqnarray*}
Proposition \ref{P7} guarantees that one can chose square integrable
functions $\Delta _{i,\eta }$ such that
\begin{eqnarray*}
&& \left\Vert \delta ^{i-1}-\Delta _{i,\eta }\right\Vert _{H_{-j}\left[ 0,T%
\right] }=O\left( \eta \right) , \\
&& \left\Vert \Delta _{i,\eta }\right\Vert _{L_{2}\left[ 0,T\right]
}=O\left( \eta^{-\frac{2i-1}{ 2\left( j-1\right) +1}} \right) ,\quad
i=1,2,...,p_{j}.
\end{eqnarray*}
Let
\begin{align*}
u_{j,\eta }=w+ & \sum_{i=1}^{p_{j}}\left( V_{i}^{0}\Delta _{i,\eta
}\left(
t\right) +V_{i}^{T}\Delta _{i,\eta }\left( t-T\right) \right)- \\
- & R_{j,j}^{-1}\sum_{j\leq i<l\leq r}\left(
Q_{j,j}B_{i,l}-B_{j,j}^{\prime }Q_{i,l}^{\prime }\right) \phi
^{i-j+1}u_{l,\eta }.
\end{align*}

This choice of $u_{j,\eta }$ guarantees $\left\Vert \gamma _{r,j}\left(
u_{\eta }-\hat{u}\right) \right\Vert _{L_{2}\left[ 0,T\right] }=O\left( \eta
\right) $ and
\begin{eqnarray*}
\left\Vert v_j - \left( w+\sum_{i=1}^{p_{j}}\left( V_{i}^{0}\Delta _{i,\eta
}\left( t\right) +V_{i}^{T}\Delta _{i,\eta }\left( t-T\right) \right)
\right)\right\Vert= \\
\hspace{1.5cm} = O\left( \sum\limits_{i=j+1}^{r}\left\Vert v_{i,\eta
}\right\Vert _{H_{-1}}\right) = O\left( \eta ^{-\frac{2p_{j}-1}{2(j-p_{j})+1}%
}+\sum\limits_{i=j+1}^{r}\left\Vert v_{i,\eta }\right\Vert _{H_{-1}}\right) .
\end{eqnarray*}%
This proves existence of a family of square-integrable controls, $\left\{
u_{\eta },~\eta >0\right\} $ such that
\begin{equation*}
\left\Vert u_{\eta }-\hat{u} \right\Vert _{\gamma _{r}}=O\left( \eta \right)
, \ \left\Vert u_{\eta }\right\Vert _{L_{2}\left[ 0,T\right] }= O\left( \eta
^{-\max_{1\leq q\leq r}\frac{2p_{q}+1}{2\left( q-p_{q}\right) -1}}\right)
\end{equation*}
and therefore,
\begin{equation*}
\sigma _{T} \leq \max \limits_{%
\begin{array}{c}
0\leq i<j\leq r, \\
(\alpha _{i,j}\neq 0\ \mathrm{or}\ \beta _{i,j}\neq 0)%
\end{array}%
} \frac{i+1/2}{2(j-i)-1} .
\end{equation*}

In order to prove the inverse inequality, pick the greatest value $m$ of
those $j$ for which $\frac{2p_{j}-1}{2(j-p_{j})+1}=\max\limits_{1\leq i\leq
r} \frac{2p_{i}-1}{2(i-p_{i})+1}$. Suppose there exists a family of controls
\begin{equation*}
\{ v_{\eta }=\left( v_{0,\eta },v_{1,\eta },\,...\,,v_{r,\eta }\right), \
\eta >0 \}
\end{equation*}
such that
\begin{equation*}
\left\Vert \hat{u}-v_{\eta }\right\Vert _{\gamma _{r}}=O(\eta );\qquad
\lim\limits_{\eta \rightarrow 0^{+}}\left\Vert v\right\Vert _{L_{2}}\eta ^{%
\frac{2p_{m}-1}{2(m-p_{m})+1}}=0.
\end{equation*}
By Proposition \ref{P7} there exist constants $C_{1},C_{2}$ such that
\begin{eqnarray*}
&& \hspace{-0.8cm} C_{1}\eta ^{-\frac{2p_{m}-1}{2(m-p_{m})+1}} \leq \\
&& \leq \Big\Vert v_{m}\! + R_{m,m}^{-1} \hspace{-0.3cm} \sum\limits_{m\leq
i<j\leq r} \hspace{-0.3cm} \left( Q_{m,m}B_{i,j}-B_{m,m}^{\prime }Q^{\prime
}_{i,j}\right) \phi ^{i-m+1}v_{j}\Big\Vert _{L_{2}} \leq C_{2}\left\Vert
v\right\Vert _{L_{2}},
\end{eqnarray*}%
which is a contradiction.

\section{Singularity of nonlinear control-affine problems}

\label{cap}

While we have managed to provide exhaustive analysis of the degree of
singularity for singular linear quadratic problems, similar questions for
control-affine nonlinear problem (\ref{j0t})-(\ref{af})-(\ref{ep}) still
need to be answered. \medskip

\noindent \textbf{Question.} \textit{Is the set of possible values of degree
of singularity $\sigma_T$ for control affine problems (\ref{j0t})-(\ref{af})-(\ref%
{ep}) finite? Is the value of $\sigma_T$ semiinteger for generic
boundary data? Does this set of numbers correspond to a (local)
stratification of the state space?} $\Box$ \medskip

There is research activity related to this question.

It is important to correlate the degree of singularity $\sigma_T$
with the order of singularity introduced by Kelley, Kopp and Moyer
in \cite{KKM}) (see also \cite{Krener}) for singular extremals of
optimal control problems.

Another recent study carried out in the context of sub-Rie\-man\-ni\-an
geometry and motion planning, invokes concepts of \textit{entropy} and
\textit{complexity} (see \cite{Jean, GaZ}). The connection of these notions
on one side with the degree of singularity on the other side is yet to be
clarified.

In the following Sections we study the degree of singularity for generalized
minimizers of control-affine problem (\ref{j0t})-(\ref{af})-(\ref{ep}) with
\textit{positive state-quadratic cost}.

Apparently commutativity/noncommutativity of controlled vector
fields affects the value of order of singularity $\sigma_T $ a great
deal. The connection between the commutativity/noncommutativity and
generalized minimizers is an established fact
\cite{Orl,Sar,Br,BRam}. Yet it is not well understood, how the Lie
structure is revealed in the properties of generalized minimizers.
The commutativity assumption is not immediately
apparent in the linear-quadratic case, but in fact, (\ref{Z13}) is \textit{%
the} commutativity condition for the class of singular problems discussed in
the previous Section.

Here is the list of results on nonlinear control-affine problems, which
appear in the following sections.

We start with control-linear (=driftless) case and prove (Section~\ref{ncdr}%
) that for generic boundary data degree of singularity equals $\frac{1}{2}$
independently of commutativity/noncommutativity of inputs.

In Section~\ref{cafm} we establish an upper estimate $\sigma_T \leq
3/2$ for the degree of singularity of control-affine problems
(\ref{j0t})-(\ref{af})-(\ref{ep}) with commuting controlled vector
fields and positive state-quadratic cost (Theorem~\ref{taf3/2}). In
Subsections~\ref{S31}-\ref{32sketch} we provide a sketch of the
proof of this result. The proof of the main result in Subsection
\ref{SS33} can be found in \cite{GS2}. A full proof for the material
in Subsection \ref{32sketch} can be found in Subsection \ref{SS4b}
below.

In Section~\ref{S8} we provide some evidence which allows us to
conjecture that the degree of singularity in the commutative case
should be $\leq 1$. An example is examined.

\section{Non-commutative driftless case: general result}

\label{ncdr}

It turns out that the non-commutative \textit{driftless} case
\begin{eqnarray}
J_{0}^{T}(u(\cdot )) &=&\int_{0}^{T}x(t)^{\prime }Px(t)dt\rightarrow
\min
,\qquad P>0,  \notag \\
\dot{x} &=&\sum_{j=1}^{r}g^{j}u_{j},\qquad x(0)=x_{0},\qquad x(T)=x_{T},
\label{colin}
\end{eqnarray}%
is rather simple. In this case for generic boundary data the value
of the order of singularity equals $\sigma _{T}=1/2$ , i.e. it does
not depend on the Lie structure of the system of vector fields
$\{g^{1},\ldots ,g^{r}\}$.

\begin{theorem}
\label{Thdriftless} Let $\mathcal{A}_{x_{0}}$ be the attainable set (in the
driftless case the orbit) of the control system (\ref{colin}). Let $x_T \in
\mathcal{A}_{x_{0}}$ and
\begin{equation}  \label{pinf}
\alpha =\inf \{x^{\prime}Px| \ x \in \mathcal{A}_{x_{0}}\} .
\end{equation}
Then:

\begin{itemize}
\item[i)] the infimum $\inf J_0^T=\alpha T$; \smallskip

\item[ii)] the degree of singularity $\sigma _{T}\geq \frac{1}{2}$ unless $
x_{0}^{\prime }Px_{0}=x_{T}^{\prime }Px_{T}=\alpha $; \smallskip

\item[iii)] if the infimum (\ref{pinf}) is attained then the degree of
singularity $\sigma _{T}\leq \frac{1}{2}.\ \Box $
\end{itemize}
\end{theorem}

A detailed proof can be found in Subsection~\ref{dlproof}. Here is a
brief idea for the case when the system $\{g^{1},\ldots ,g^{r}\}$
has complete Lie rank. Then, roughly speaking, generalized optimal
trajectory consists of three 'pieces': an initial 'jump', which
brings it to the origin $0$, a constant piece $x(t)\equiv 0,\ t\in
]0,T[$, and a final 'jump' to the end point $x(T)=x_{T}$. Evidently
$\inf J_{0}^{T}=0$ and a simple homogeneity based argument shows
that $\sigma _{T}\leq 1/2$.

To prove that in fact $\sigma _{T}=1/2$ whenever $x_{0}\neq 0$ or $x_{T}\neq 0$%
, assume that $x_{0}\neq 0$ (the case $x_{T}\neq 0$ is analogous), fix $%
\varepsilon >0$ and take a control $u_{\varepsilon }$, such that
\begin{equation}
J_{0}^{T}(u_{\varepsilon })=\int_{0}^{T}x_{u_{\varepsilon }}^{\prime
}(t)Px_{u_{\varepsilon }}(t)dt<\inf J_{0}^{T}+\varepsilon =\varepsilon .
\label{peps}
\end{equation}%
Consider the set $\left\{x\in I\!\!R^{n}|\ x^{\prime }Px\leq \frac 1
2 x_{0}^{\prime
}Px_{0} \right\}$; let $\rho $ be the distance from $x_{0}$ to this set. Since $%
x^{\prime }Px$ is positive definite, one concludes from the
inequality (\ref{peps}), that there exists $t_{\varepsilon
}<\frac{2\varepsilon }{ x_{0}^{\prime }Px_{0}}$ such that
$x_{u_{\varepsilon }}(t_{\varepsilon })^{\ast }Px_{u_{\varepsilon
}}(t_{\varepsilon })<(1/2)x_{0}^{\prime }Px_{0}$. Then, by
Cauchy-Schwarz inequality, the control needed to achieve $
x_{u_{\varepsilon }}(t_{\varepsilon })$ from $x_{0}$ in time
$t_{\varepsilon }$, must satisfy the estimate $\Vert u_{\varepsilon
}\Vert _{L_{2}[0,T]}^{2}\geq \rho t_{\varepsilon }^{-1}\geq
C\varepsilon ^{-1}$, for some $C>0$.

\section{Degree of singularity for control-affine systems: commuting inputs}

\label{cafm}

In this section we discuss the degree of singularity of optimal
control problems of type (\ref{j0t})-(\ref{af})-(\ref{ep}) with $T <
+ \infty $.

In what follows we denote by $e^{tF}$ the flow of the
smooth field $F$; for each point $x_{0}\in \mathbb{R}^{n}$ the curve $%
t\mapsto e^{tF}x_{0}$ is the unique solution of the differential
equation
\begin{equation*}
\dot{x}=F(x),\qquad x(0)=x_{0}.
\end{equation*}%
For fixed $t$ the map $ x \mapsto e^{tF}x $ is a local
diffeomorphism in a neighborhood of any point $x_0$ such that
$e^{tF}x_0 $ exists.

For every (local) diffeomorphism $P:\mathbb{R}^{n}\mapsto
\mathbb{R}^{n}$, and any vector field $F$, we denote by $AdPF$ the
(local) field defined as
\begin{equation*}
AdPF(x)=(DP(x))^{-1}F(P(x)),
\end{equation*}%
where $DP(x)$ denotes the Jacobian matrix of $P$ evaluated at the
point $x$.

We keep the notation introduced in Section \ref{S2}, according to which $%
\phi $ denotes primitivization (i.e., $\phi u\left( t\right)
=\int_{0}^{t}u\left( \tau \right) ~d\tau ,\forall u\in L_{1,loc}$).

We introduce several assumptions.

\begin{assumption}
\label{as3} The fields $f$, $g_i$, $i=1,...,k$ are complete. The
controlled vector fields $g_i$ span $k$-dimensional involutive
distribution; for simplicity we assume that $[ g_i,g_j ] \equiv 0$
holds for all $i,j$. $\square$
\end{assumption}

The following three assumptions regard conditions on the growth of
$f$, $g^i$ and of their derivatives.

\begin{assumption}
\label{as21}For any compact set $K\subset \mathbb{R}^{n}$
\begin{equation*}
\lim\limits_{|v|\rightarrow +\infty }\left\vert e^{Gv}x\right\vert
=+\infty ,
\end{equation*}
uniformly with respect to $x\in K$.$\ \square $
\end{assumption}

\begin{assumption}
\label{as2}For any compact set $K\subset \mathbb{R}^{n}$
\begin{equation*}
\lim\limits_{|v|\rightarrow +\infty }\frac{\left\vert \frac{\partial
}{\partial x}\left( \left( e^{Gv}x\right) ^{\prime }P\left(
e^{Gv}x\right) \right) \right\vert }{\left\vert e^{Gv}x\right\vert
^{2}}=0,
\end{equation*}%
uniformly with respect to $x\in K$.$\ \square $
\end{assumption}

\begin{assumption}
\label{as1} For each compact set $K\subset \mathbb{R}^{n}$ there
exists a function $\gamma :[0,+\infty \lbrack \mapsto \mathbb{R}$
bounded below, such that:

\begin{itemize}
\item[\textit{i)}] $\lim \limits _{s \rightarrow + \infty} \frac{\gamma (s)}{%
s} = + \infty;$ \smallskip

\item[\textit{ii)}] $%
\begin{array}[t]{l}
\left| e^{Gv} x \right|^2 \geq \gamma \left( \left| \left( Ad \left(
e^{Gv} \right) f \right) (x) \right|+ \left| \frac{\partial
}{\partial x} \left( Ad
\left( e^{Gv} \right) f \right) (x) \right| \right), \medskip \\
\hspace{7cm} \forall (x,v) \in K \times \mathbb{R}^k . \ \square%
\end{array}%
$
\end{itemize}
\end{assumption}

As we will see below, Assumption \ref{as3} allows us to use a
reduction procedure analogous to the procedure employed in Section
\ref{S2} for the treatment of singular linear-quadratic problem,
while Assumptions~\ref{as21},\ref{as2} and \ref{as1} guarantee
existence of minimizers in a suitable class of generalized controls.

The Assumptions~\ref{as2} and \ref{as1} are somehow less explicitly
formulated. In the two following remarks we formulate more
particular growth conditions onto vector fields $f$ and $g^i$ and
their flows which guarantee fulfillment of the two Assumptions.

\begin{remark}
Assume that the drift vector field satisfies the growth condition
\[
    |f(y)| \leq \psi(|y|).
\]
For the flow $e^{Gv}$ generated by the  controlled vector fields, we
require
\begin{eqnarray}
  \left|De^{Gv}x\right| &=& o\left(\left|e^{Gv}x \right| \right), \label{de} \\
  \left|\left(De^{Gv}x \right)^{-1}\right| &=&
  o\left( \frac{\left|e^{Gv}x \right|^2}{\psi\left(|e^{Gv}x |\right)} \right), \label{de-1}\\
\left|D^2e^{Gv}x\right| &=&
O\left(\frac{\psi\left(|e^{Gv}x|\right)}{\left|e^{Gv}x\right|^3}
\right), \label{d2e}
\end{eqnarray}
as $|v| \rightarrow \infty$, uniformly for $x$ belonging to any
fixed compact $K$.

A typical choice of $\psi$ which guarantees completeness of $f$
would be $$\psi(|y|)=k(1+|y|).$$ In this case (\ref{de}) and
(\ref{de-1}) take the form
\[
\left|De^{Gv}x\right|+\left|\left(De^{Gv}x\right)^{-1}\right|
=o\left(\left|e^{Gv}x\right| \right),
\]
while (\ref{d2e}) would mean that $\left|D^2e^{Gv}x\right|
\left|e^{Gv}x\right|^2$ is uniformly (with respect to $x$ from a
fixed compact $K$) upper bounded for all $v. \ \square$
\end{remark}

Another case we had in mind is described in the following

\begin{remark}
In the particular case of constant vector fields  $g_{i},\
i=1,2,...,k$ the condition \textit{(ii)} of the Assumption \ref{as1}
reads
\begin{equation*}
|v|^{2}\geq \gamma \left( \left\vert f(x+Gv)\right\vert +\left\vert
Df(x+Gv)\right\vert \right) ,\qquad \forall (x,v)\in K\times
\mathbb{R}^{k},
\end{equation*}%
i.e. $|f|$ and $|Df|$ must exhibit subquadratic growth along the
directions spanned by $g_i, \ i=1,2,...,k$.

It is straightforward to check that Assumptions \ref{as21} and
\ref{as2} hold if the fields $g_i, \ i=1,2,...,k$ are linearly
independent. $\square $
\end{remark}

Our main result in the commutative case is the following.

\begin{theorem}
\label{taf3/2}If Assumptions \ref{as3}, \ref{as21}, \ref{as2} and
\ref{as1} hold for problem (\ref{j0t})-(\ref{af})-(\ref{ep}), then
\begin{equation*}
\sigma _{T}\leq \frac{3}{2}
\end{equation*}
for $T<+\infty $ and generic boundary conditions.$\ \square $
\end{theorem}

The rest of this Section contains sketch of the proof of Theorem
\ref{taf3/2}. The feature which distinguishes the proof is an
\textit{unbounded} set of control parameters. In this context some
components of the proof gain (in our opinion) an independent
interest. Among those is Theorem~\ref{exlr} on existence  and
Lipschitzian regularity of relaxed minimizing trajectories in the
case of unbounded controls. We discuss this question in details in
\cite{GS2}. Another important issue is Proposition~\ref{P5} on
approximation of relaxed trajectories by ordinary ones and on
estimates of the variation of the approximants; this result is
proved in Subsection \ref{SS4b}.

\subsection{Proof of Theorem~\protect\ref{taf3/2}: desingularization}

\label{S31}

First we proceed with a "desingularizing transformation", which
appeared in \cite{AgS} under the name of 'reduction' and proved to
be useful for analysis of control-affine systems with unconstrained
controls.

\begin{proposition}
\label{p reduction} Under Assumption \ref{as3} the following holds:
\begin{equation*}
x_u (t) = e^{G \phi u (t) } y_{\phi u } (t), \qquad \forall t \in [0,T], \ u
\in L_\infty [0,T] ,
\end{equation*}
where $x_u$ denotes the trajectory of the system
\begin{equation*}
\dot x (t) = f(x(t)) +G(x(t))u(t), \qquad x(0) = x_0 ,
\end{equation*}
$y_v$ denotes the trajectory of the system
\begin{equation*}
\dot y (t) = \left( Ad \left( e^{Gv(t)}\right) f \right) (y(t)), \qquad y(0)
= x_0. \ \square
\end{equation*}
\end{proposition}

The substitution $x = e^{G \phi u (t) } y $ (sometimes called in the
literature \textit{Goh transform}) leads us to the following
'desingularized' problem\footnote{An anonymous referee brought to
our attention a publication by Yu. Orlov \cite{Or85} where such
transform has been introduced (without Lie algebraic notation) with
the scope of passing from control problems with measure-like
controls to classical control problems. This construction is
described by the same author in  \cite[Ch. 4]{Or1}.}
\begin{eqnarray}  \label{raf1}
&& \dot y (t) = \left( Ad \left( e^{Gv(t)}\right) f \right) (y(t)) , \\
&& J_r(v) = \int_0^T \left( e^{G v (t) } y(t) \right)^\prime P \left( e^{G v
(t) } y(t) \right) dt \rightarrow \min ,  \label{rq}
\end{eqnarray}
with boundary conditions
\begin{equation}  \label{rep}
y(0)=x_0, \qquad y(T)=e^{GV} x_T, \quad V \in \mathbb{R}^k.
\end{equation}

Notice that in the case when the fields $g_i, \ i=1,2,...,k$ are constant we
have $e^{Gv}y = y+Gv$, $\left( Ad \left( e^{Gv} \right) f \right) (y) = f(y+
Gv) $. Therefore, the reduced cost functional (\ref{rq}) takes the form

\begin{equation}  \label{rqc}
J_r(v) = \int_0^T y(t)^\prime P y(t) + 2 v(t)^\prime G^\prime P y(t)
+ v(t)^\prime G\prime P G v(t) \, dt .
\end{equation}
Thus, in this case the reduction procedure achieves
desingularization in the sense that the integrand in (\ref{rqc})
exhibits quadratic growth with respect to the new control $v$
(provided $G$ is full rank). This does not hold in the general case
with nonconstant $G$, but it can be reasonably expected that it is a
fairly generic outcome. However, one important feature of the
nonlinear optimal control problem
(\ref{raf1})-(\ref{rq})-(\ref{rep}) is its lack of convexity with
respect to $v$. If we introduce the notation
\begin{equation}  \label{laf}
\left( \tilde f_0 (y,v), \tilde f (y,v) \right) = \left( \left( e^{G v } y
\right)^\prime P \left( e^{G v } y \right) , \left( Ad \left( e^{Gv} \right)
f \right) (y) \right),
\end{equation}
then, for generic $y \in \mathbb{R }^n$ (fixed) the set
\begin{equation*}
\Gamma(y)=\left\{ \left( y_0, \tilde f (y,v) \right) : y_0 \geq \tilde f_0
(y,v), \ v \in \mathbb{R}^m \right\} \subset \mathbb{R}\times \mathbb{R}^n
\end{equation*}
is in general nonconvex, even in the case when $G$ is constant. For
example, for scalar $v$ the set $\left\{ \tilde f(y,v), v \in
\mathbb{R }\right\}$ is a curve in $\mathbb{R}^n$.

It follows that classical minimizers for the reduced problem (\ref{raf1})-(%
\ref{rq})-(\ref{rep}) typically fail to exist. Instead, existence of \textit{%
relaxed} minimizers can be established under the Assumptions
\ref{as3}, \ref{as21}, \ref{as2} and \ref{as1}.

\subsection{Proof of Theorem~\protect\ref{taf3/2}: digression on relaxed
controls}

\label{dig}

Recall that a relaxed control can be seen as a family $t\mapsto \eta _{t}$
of probability measures in the space of $v\in \mathbb{R}^{k}$, such that $%
t\mapsto \eta _{t}$ is measurable in the weak sense with respect to $t\in
\lbrack 0,T]$.

\begin{definition}
\label{D1}Consider a nonempty set $A\subset \mathbb{R}^{k}$. We denote by $%
\mathcal{M}_{A}$ the set of inner regular probability measures in $\mathbb{R}%
^{k}$ with compact support contained in $A$. We denote by $\mathcal{M}_{A}%
\left[ 0,T\right] $\ the set of functions $t\mapsto \eta _{t}$, $t\in \left[
0,T\right] $, where:

\begin{itemize}
\item[i)] $\eta _{t} \in \mathcal{M}_A, \qquad \forall t \in [0,T]$;
\smallskip

\item[ii)] For any continuous function, $g:\mathbb{R}^{1+k}\mapsto \mathbb{R}
$, the function $t\mapsto \int_{\mathbb{R}^{k}}g\left( t,u\right) \
\eta _{t}\left( du\right) $ is measurable. $\square$
\end{itemize}
\end{definition}

Let $\delta _{v}$ denote the Dirac measure in $\mathbb{R}^{k}$ concentrated
at the point $v$. Obviously, the space of measurable essentially bounded
functions can be embedded in the space $\mathcal{M}_{\mathbb{R}^{k}}\left[
0,T\right] $ through the map $v\left( t\right) \mapsto \delta _{v\left(
t\right) }$.

The dynamic equation (\ref{raf1}) and cost functional (\ref{rq})
corresponding to a relaxed control become
\begin{eqnarray*}
&&\dot{y}(t)=\int_{\mathbb{R}^{k}}\left( Ad\left( e^{Gv}\right) f\right)
(y(t))\,\eta _{t}(dv), \\
&&J_{r}(\eta )=\int_{0}^{T}\int_{\mathbb{R}^{k}}\left( e^{Gv}y(t)\right)
^{\prime }P\left( e^{Gv}y(t)\right) \,\eta _{t}(dv)\,dt\rightarrow \min .
\end{eqnarray*}%
It is clear that the equations above coincide with (\ref{raf1}) and
(\ref{rq}) in the case when $\eta _{t}=\delta _{v\left( t\right) }$
holds at almost every $t\in \left[ 0,T\right] $ for some function
$v\in L_{\infty }^{k}\left[ 0,T\right] $.

We use the short notation $\left\langle \eta , f \right\rangle $ to
indicate the averaging of $f$ by the measure $\eta$. I.e., for $g:
\mathbb R^{1+k} \mapsto \mathbb R$, $\eta \in \mathcal M _A [0,T]$,
$ \left\langle \eta_t, g(t, \cdot ) \right\rangle $ denotes the
function \[ \left\langle \eta_t, g(t, \cdot ) \right\rangle = \int
_{\mathbb R^k} g(t,u) \eta_t (du) , \qquad t \in [0,T]. \] The
following proposition (see \cite{Gam}) relates relaxation with the
convexification of the right-hand side.

\begin{proposition}
\label{P12}Consider a $C_{1}$-map $X:\mathbb{R}^k \mapsto \mathbb{R}^{n}$,
and a nonempty set $A \subset \mathbb{R}^k$. Then,
\begin{equation}  \label{chul}
\left\{ \langle \eta , X \rangle , \ \eta \in \mathcal{M}_A\right\}
=conv\left\{ X\left( v\right) ,~v\in A \right\} .\square
\end{equation}
\end{proposition}

Below we will consider a more restricted class of controls (called
sometimes {\it Gamkrelidze relaxed or chattering controls}) of
special form
\begin{equation}
\eta _{t}=\sum_{j=1}^{N}p_{j}(t)\delta _{v^{j}(t)},\qquad p_{j}(t)\geq
0,\qquad \sum_{j=1}^{N}p_{j}(t)\equiv 1.  \label{gco}
\end{equation}%
It is clear that any measurable essentially bounded control can be
identified with such a  control with $N=1$. The dynamic equation and
functional corresponding to these controls become
\begin{eqnarray}
&&\dot{y}(t)=\sum_{j=1}^{N}p_{j}(t)\left( Ad\left( e^{Gv^{j}(t)}\right)
f\right) (y(t));  \label{rel} \\
&&J_{r}(\eta )=\int_{0}^{T}\sum_{j=1}^{N}p_{j}(t)\left(
e^{Gv^{j}(t)}y(t)\right) ^{\prime }P\left( e^{Gv^{j}(t)}y(t)\right)
\,dt\rightarrow \min .  \label{relJr}
\end{eqnarray}

Looking at Proposition~\ref{P12} it is easy to understand why the class of
controls (\ref{gco}) suffices: each vector of the convex hull in (\ref{chul}%
) can be represented as a finite convex combination of $\left( Ad\left(
e^{Gv^{j}}\right) f\right) \left( y\right) $; the same is valid for the
functional. By virtue of Carath\'{e}odory theorem the combination need not
have more than $n+2$ summands.

\subsection{Existence of relaxed minimizers and Lipschitzian regularity of
optimal relaxed trajectories\label{SS33}}

The existence of minimizing relaxed controls in the case where the
set of control parameters is bounded is a well known  result closely
related to classical A.F.Filippov's existence theorem (see
\cite[Ch.8]{Gam}). Our treatment involves controls without any
\textit{a priori} bound. Existence results for relaxed minimizers in
this case, are referred to in \cite[Ch. 11]{Ces}, but we are not
aware of any previous results on Lipschitzian regularity of relaxed
minimizing trajectories. In a separate paper \cite{GS2} we present a
technically involved proof of Lipschitzian regularity of relaxed
minimizers. In our present setting the result in \cite{GS2} takes
the special form:

\begin{theorem}
\label{exlr} Under Assumptions \ref{as3}, \ref{as21}, \ref{as2} and
\ref{as1}, the reduced problem (\ref{raf1})-(\ref{rq})-(\ref{rep})
admits a  minimizer, $\eta
_{t}=\sum\limits_{j=1}^{n+2}p_{j}(t)\delta _{v^{j}(t)}$. Any such
minimizer satisfies the Pontryagin maximum principle and, provided
it does not correspond to a strictly abnormal extremal, there exists
a constant $M<+\infty $ such that
\begin{eqnarray}
&&\left\vert \left( Ad\left( e^{Gv^{j}(t)}\right) f\right) (y_{\eta
}(t))\right\vert \leq M,  \label{lif} \\
&&\left( e^{Gv^{j}(t)}y_{\eta }(t)\right) ^{\prime }P\left(
e^{Gv^{j}(t)}y_{\eta }(t)\right) \leq M\   \label{liJ}
\end{eqnarray}%
hold for $j=1,2,...,n+2$ and almost every $t\in \lbrack 0,T]$. $\square $
\end{theorem}

\begin{remark}
\label{R1}For generic boundary conditions the relaxed optimal trajectory
corresponds to a normal extremal. Therefore, it follows from inequality (\ref%
{lif}) that generic optimal relaxed trajectories are Lipschitz
continuous. Inequality (\ref{liJ}) together with Assumption
\ref{as2} imply that the functions $v^{j}$, $j=1,2,...,n+2$ on the
corresponding Gamkrelidze control are essentially bounded.$\square $
\end{remark}

\subsection{Approximation of relaxed minimizers by absolutely-continuous
controls}

\label{32sketch}

The rest of the proof of Theorem \ref{taf3/2} consists of two approximation
steps. In the first step we approximate the relaxed minimizer $\eta _{t}$ of
the reduced problem (\ref{raf1})-(\ref{rq})-(\ref{rep}) by a piecewise
constant control $w_{\varepsilon }(\cdot )$ in such a way that the
trajectory and the functional of (\ref{raf1}), driven by $w_{\varepsilon
}(\cdot )$, are $\varepsilon $-close to the trajectory and the functional of
(\ref{rel}). Such approximating controls can be chosen in such a way that
the number of discontinuities is bounded by $\frac{Const.}{\varepsilon }$.

The second step consists of approximating the piece-wise constant controls $%
w_{\varepsilon }(\cdot )$ by \textit{absolutely continuous} controls $%
v_{\varepsilon }(\cdot )$, whose derivative $u_{\varepsilon }(\cdot )=\dot{v}%
_{\varepsilon }(\cdot )$ becomes $\varepsilon $-minimizing control for the
original problem (\ref{j0t})-(\ref{af})-(\ref{ep}). This second approximants
can be obtained by altering the function $w_{\varepsilon }$ at intervals of
length $\varepsilon ^{2}$ containing the points of discontinuity.

Using this two-step approximation we are able to prove the following
proposition (see Subsection \ref{SS4b} below).

\begin{proposition}
\label{P5}Consider $\eta _{t}=\sum\limits_{j=1}^{n+2}p_{j}(t)\delta
_{v^{j}(t)}$, a Gamkrelidze minimizer for the reduced problem (\ref{raf1})-(%
\ref{rq})-(\ref{rep}), satisfying (\ref{lif}), (\ref{liJ}) for some $%
M<+\infty $, let $y_{\eta }$ denote the corresponding solution to the
dynamical equation (\ref{rel}).

There exists a family of absolutely continuous piecewise linear controls $%
v_\varepsilon$, $\varepsilon >0$ such that

\begin{itemize}
\item[i)] $J_r \left( v_\varepsilon \right) = J_r (\eta) + O(\varepsilon)$
when $\varepsilon \rightarrow 0^+ $; \smallskip

\item[ii)] The trajectories of the reduced system (\ref{raf1}) satisfy $%
\left\Vert y_{v_{\varepsilon }}-y_{\eta }\right\Vert _{L_{\infty
}[0,T]}=O(\varepsilon )$ when $\varepsilon \rightarrow 0^{+}$; \smallskip

\item[iii)] The trajectories of the original system (\ref{af}) corresponding
to the controls $u_{\varepsilon }=\frac{d}{dt}v_{\varepsilon }$ satisfy $%
x_{u_{\varepsilon }}(T)=x_{T}+O(\varepsilon )$ when $\varepsilon \rightarrow
0^{+}$;

\item[iv)] $\left\Vert \frac{d}{dt}v_{\varepsilon }\right\Vert
_{L_{2}[0,T]}^{2}=O\left( \frac{1}{\varepsilon ^{3}}\right) $ when $%
\varepsilon \rightarrow 0^{+}$. $\square $
\end{itemize}
\end{proposition}

The proof of Theorem \ref{taf3/2} follows by noting that
\begin{equation}
\sigma _{T}\leq \lim\limits_{\varepsilon \rightarrow 0^{+}}\frac{\ln
\left( Const.\left\Vert u_{\varepsilon }\right\Vert
_{L_{2}[0,T]}\right) }{\ln \frac{1}{\varepsilon }},  \label{nv}
\end{equation}%
where $u_{\varepsilon }=\frac{d}{dt}v_{\varepsilon }$ is the family
of controls described in Proposition \ref{P5}. This yields the
estimate $\sigma _{T}\leq \lim\limits_{\varepsilon \rightarrow
0^{+}}\frac{Const+\ln \varepsilon ^{-\frac{3}{2}}}{\ln
\frac{1}{\varepsilon }}=\frac{3}{2}.$

\section{Degree of singularity for input-commutative control-affine system:
conjecture and example}\label{S8}

In the previous Section we provide an upper bound for the degree of
singularity by showing how to construct a minimizing sequence with
asymptotics $\sigma_T = \frac 3 2$. However we believe that this
upper bound is not sharp and we provide the following conjecture for
a sharp estimate:

\begin{conjecture}
\label{conjex}Under the assumptions of Theorem~\ref{taf3/2}
\begin{equation*}
\sigma _{T}\leq 1.\ \Box
\end{equation*}
\end{conjecture}

Our conjecture relies on the proof of Proposition~\ref{P5}, which is key
fragment of the proof of Theorem~\ref{taf3/2}. We trust that our two-step
approximation procedure can be improved: there exists a piecewise continuous
control $w_{\varepsilon }(\cdot )$ with $\leq O(\varepsilon ^{-1})$
intervals of continuity, such that the \textit{end-point} of the trajectory
and the value of the functional of (\ref{rq}) driven by $w_{\varepsilon
}(\cdot )$ are \textit{$\varepsilon ^{2}$-close} to the \textit{end-point}
of the trajectory and the value of the functional of (\ref{relJr}).

If this holds true then, by modifying this approximant in intervals
of length $\varepsilon ^{3}$ instead of $\varepsilon ^{2}$ we obtain
a family of square-integrable controls $u_{\varepsilon
}=\frac{dv_{\varepsilon }}{dt}$ satisfying the estimate $\Vert
u_{\varepsilon }\Vert _{L_{2}}=O(\varepsilon ^{-2})$. Then, by
virtue of majoration (\ref{nv}) we conclude that $\sigma _{T}\leq
1$.

Another possibility for sharpening the upper estimate of degree of
singularity is related to the second approximation step described in the
previous subsection. This step can be formalized as the following problem of
best approximation.

\begin{problem}
Let $B_{M}=\{u:\Vert u\Vert _{L_{2}[0,T]}\leq M\}$ denote the ball of radius
$M$ in the space of square-integrable functions. Given a
piecewise-continuous (or just essentially bounded) function $\varphi
:[0,T]\mapsto \mathbb{R}$, find the asymptotics (the rate of decay) of the
distance
\begin{equation*}
\rho _{L_{p}}(\varphi ,B_{M})=\inf \left\{ \left\Vert \varphi -\phi
u\right\Vert _{L_{p}[0,T]}:u\in B_{M}\right\} ,
\end{equation*}%
as $M\rightarrow +\infty $ (for fixed $p\in \mathbb{N}$).\ $\Box $
\end{problem}

The following example shows that, at least in some cases, the bound
$\sigma _{T}\leq 1$ is tight.

\begin{example}
\label{ex1} Consider optimal control-affine problem
\begin{eqnarray*}
&&\dot{x}=f(x)+g^{1}(x)u_{1},\qquad x=(x_{1},x_{2},x_{3}), \\
&&f=x_{1}\frac{\partial }{\partial x_{2}}+\gamma (x_{1})(x_{1}^{2}-1)\frac{%
\partial }{\partial x_{3}},\qquad g_{1}=\frac{\partial }{\partial x_{1}} \\
&&J_{0}^{1}=\int_{0}^{1}(x_{1}^{2}+x_{2}^{2}+x_{3}^{2})dt\rightarrow \min ,
\\
&&x(0)=0,\qquad x(1)=0,
\end{eqnarray*}%
where $\gamma (x)$ is a smooth function supported at $[-2,2]\subset \mathbb{R%
}$, $0\leq \gamma (x)\leq 1$ and $\gamma (x)\equiv 1$ on $[-3/2,3/2].\ \Box $
\end{example}

In coordinates the dynamics of the problem is
\begin{equation}
\dot{x}_{1}=u_{1},\qquad \dot{x}_{2}=x_{1}\qquad \dot{x}_{3}=\gamma
(x_{1})(x_{1}^{2}-1).  \label{dyex1}
\end{equation}

To estimate the infimum of this problem note that
\begin{equation*}
0=\int_{0}^{1}\gamma (x_{1}\left( t\right) )\left( x_{1}\left( t\right)
^{2}-1\right) \,dt\leq \int_{0}^{1}x_{1}(t)^{2}\,dt-1,
\end{equation*}%
and hence $J_{0}^{1}\geq 1$. Now we construct a minimizing sequence
of controls $u_{N}(\cdot )$, such that $J_{0}^{1}(u_{N})\rightarrow
1$ as $ N\rightarrow +\infty $.

First take the indicator function $p(t)$ of the interval $[0,1]$ and
construct a piecewise-constant function
\begin{equation*}
q_{N}(t)=\sum_{j=0}^{2N-1}(-1)^{j}p(2Nt-j);
\end{equation*}%
$N$ being a large integer. Its intervals of constancy have lengths equal to $%
(2N)^{-1}$.

Then, alter the function $q_{N}$ on the subintervals $[0,N^{-3}]$,
\linebreak $\lbrack j(2N)^{-1}-N^{-3},j(2N)^{-1}+N^{-3}],\quad j=1,\ldots
,2N-1\ $and $[1-N^{-3},1]$, transforming it into a piecewise-linear
continuous function $q_{N}^{c}(t)$ with boundary values: $%
q_{N}^{c}(0)=q_{N}^{c}(T)=0$.

Taking $x_{1}(t)=q_{N}^{c}(t)$ and substituting it into second and third
equations of (\ref{dyex1}), we conclude that the corresponding solution
satisfies the conditions $x_{2}(1)=0,\ x_{3}(1)=O(N^{-2})$, as $N\rightarrow
+\infty $. Besides
\begin{equation*}
\int_{0}^{1}(x_{1}^{2}+x_{2}^{2}+x_{3}^{2})dt-1=O(N^{-2}),\qquad \mathrm{as}%
~N\rightarrow +\infty .
\end{equation*}%
The $L_{2}$-norm of the corresponding control $\Vert u_{1}(t)\Vert
_{L_{2}}=\Vert \dot{q}_{N}^{c}(t)\Vert _{L_{2}}$ admits an estimate
\begin{equation*}
\Vert u_{1}(t)\Vert _{L_{2}}\simeq ((N^{3})^{2}N^{-3}N)^{1/2}=N^{2}
\end{equation*}%
as $N\rightarrow +\infty $. Therefore the order of singularity satisfies $%
\sigma _{1}\leq 1.\ \Box $

\section{Proofs\label{S4}}

\subsection{Proof of Proposition \protect\ref{P7}\label{SS4a}}

In order to prove Proposition \ref{P7}, we will use the following variant of
Tchebyshev inequality:

\begin{lemma}
\label{L2}Let $u\in L_{2}$ and $\lambda $ denote Lebesgue measure.
Then
\begin{equation}  \label{tchin}
\left\Vert u \right\Vert _{L_{2}}<\eta \Rightarrow \forall
\varepsilon >0: \lambda \left\{ x:\left\vert u\left( x\right)
 \right\vert \geq \varepsilon \right\} <\frac{\eta
^{2}}{\varepsilon ^{2}}.\square
\end{equation}
\end{lemma}

\smallskip

It suffices to prove Proposition \ref{P7} for distributions $v=\delta
^{(m-1)}$.

Notice that for any $p \geq m$:
\begin{equation*}
\phi ^{p}\delta ^{(m-1)}\left( t\right) =\frac{t^{p-m}}{\left(
p-m\right) !} ,\qquad a.e.~t\in \left[ 0,T\right] .
\end{equation*}
Fix $\eta >0$, and consider $u\in L_{2}\left[ 0,T\right] $ such that
$ \left\Vert u-v\right\Vert _{H_{-p}\left[ 0,T\right] }<\eta $. It
follows from (\ref{tchin}) that $\lambda \left\{ t\in \left[
0,T\right] :\left\vert \phi ^{p}u\left( t\right)
-\frac{t^{p-m}}{\left( p-m\right) !}\right\vert \geq \varepsilon
\right\} <\frac{\eta ^{2}}{\varepsilon ^{2}}$. Then, for every
$\theta >0$ sufficiently small and provided $\eta >0$ is small,
there
exist $\theta _{0,i}\in \left] i\theta ,i\theta +\frac{\eta ^{2}}{%
\varepsilon ^{2}}\right[ $, $i=1,2,...,2^{p-m}$ such that $\left\vert \phi
^{p}u\left( \theta _{0,i}\right) -\frac{\theta _{0,i}^{p-m}}{(p-m)!}%
\right\vert <\varepsilon $.

By the mean-value theorem, there exist
\begin{equation*}
\theta _{1,i}\in \left] \left( 2i-1\right) \theta ,2i\theta +\frac{\eta ^{2}%
}{\varepsilon ^{2}}\right[, \ i=1,2,...,2^{p-m-1}
\end{equation*}
such that
\begin{eqnarray*}
&& \hspace{-0.5cm} \left\vert \phi ^{p-1}u\left( \theta
_{1,i}\right) -\frac{\theta
_{1,i}^{p-m-1}}{(p-m-1)!}\right\vert = \\
&& =\frac{\left\vert \phi ^{p}u\left( \theta _{0,2i}\right)
-\frac{\theta
_{0,2i}^{p-m}}{(p-m)!}-\left( \phi ^{p}u\left( \theta _{0,2i-1}\right) -%
\frac{\theta _{0,2i-1}^{p-m}}{(p-m)!}\right) \right\vert }{\theta
_{0,2i}-\theta _{0,2i-1}} \leq \frac{2\varepsilon }{\theta
-\frac{\eta ^{2}}{ \varepsilon ^{2}}}.
\end{eqnarray*}
Proceeding by induction we establish existence of $\theta _{p-m}\in \left]
\theta ,2^{p-m}\theta +\frac{\eta ^{2}}{\varepsilon ^{2}}\right[ $ such that
\begin{equation*}
\left\vert \phi ^{m}u\left( \theta _{p-m}\right) -1\right\vert
<\frac{ 2^{p-m}\varepsilon }{\left( \theta -\frac{\eta
^{2}}{\varepsilon ^{2}} \right) ^{p-m}}.
\end{equation*}
This implies
\begin{equation*}
\phi ^{m}u\left( \theta _{p-m}\right) >1-\frac{2^{p-m}\varepsilon }{\left(
\theta -\frac{\eta ^{2}}{\varepsilon ^{2}}\right) ^{p-m}}.
\end{equation*}
Once again, mean-value theorem guarantees the existence of $\theta
_{p-m+1}\in \left] 0,2^{p-m}\theta +\frac{\eta ^{2}}{\varepsilon
^{2}}\right[ $ such that
\begin{equation*}
\phi ^{m-1}u\left( \theta _{p-m+1}\right) =\frac{\phi ^{m}u\left( \theta
_{p-m}\right) }{\theta _{p-m}} \geq \frac{1}{2^{p-m}\theta +\frac{\eta ^{2}}{
\varepsilon ^{2}}}\left( 1-\frac{2^{p-m}\varepsilon }{\left( \theta -\frac{
\eta ^{2}}{\varepsilon ^{2}}\right) ^{p-m}}\right) .
\end{equation*}
Repeating the same argument, one proves existence of $\theta
_{p-1}\in \left] 0,2^{p-m}\theta +\frac{\eta ^{2}}{\varepsilon
^{2}}\right[ $ such that
\begin{equation*}
\phi u\left( \theta _{p-1}\right) \geq \frac{1}{\left( 2^{p-m}\theta +\frac{%
\eta ^{2}}{\varepsilon ^{2}}\right) ^{m-1}}\left( 1-\frac{2^{p-m}\varepsilon
}{\left( \theta -\frac{\eta ^{2}}{\varepsilon ^{2}}\right) ^{p-m}}\right) .
\end{equation*}%
Applying Schwarz's inequality, we conclude
\begin{eqnarray*}
\frac{1}{\left( 2^{p-m}\theta +\frac{\eta ^{2}}{\varepsilon ^{2}}\right)
^{m-1}}\left( 1-\frac{2^{p-m}\varepsilon }{\left( \theta -\frac{\eta ^{2}}{%
\varepsilon ^{2}}\right) ^{p-m}}\right) \leq \\
\leq \sqrt{\theta _{p-1}}\sqrt{\int_{0}^{\theta _{p-1}}u\left( \tau \right)
^{2}~d\tau }<\sqrt{2^{p-m}\theta +\frac{\eta ^{2}}{\varepsilon ^{2}}}%
\left\Vert u\right\Vert _{L_{2}\left[ 0,T\right] },
\end{eqnarray*}%
and
\begin{equation}
\left\Vert u\right\Vert _{L_{2}\left[ 0,T\right] }\geq \frac{1}{\left(
2^{p-m}\theta +\frac{\eta ^{2}}{\varepsilon ^{2}}\right) ^{m-\frac{1}{2}}}%
\left( 1-\frac{2^{p-m}\varepsilon }{\left( \theta -\frac{\eta ^{2}}{%
\varepsilon ^{2}}\right) ^{p-m}}\right) .  \label{Z4}
\end{equation}%
Now, for $\varepsilon =\eta ^{\frac{2(p-m)}{2(p-m)+1}}$, $\theta
=2^{p-m+1}\eta ^{\frac{2}{2(p-m)+1}}$, inequality (\ref{Z4}) reduces to
\begin{equation*}
\left\Vert u\right\Vert _{L_{2}\left[ 0,T\right] }\geq \frac{1}{\left(
2^{p-m+2}+1\right) ^{m-\frac{1}{2}}\eta ^{\frac{2m-1}{2(p-m)+1}}}.
\end{equation*}%
This proves existence of a constant $C>0$ such that
\begin{equation*}
\frac{\inf \left\{ \log \left\Vert u\right\Vert _{L_{2}\left[ 0,T\right]
}:\left\Vert u-v\right\Vert _{H_{-p}\left[ 0,T\right] } < \eta \right\} }{%
\log \frac{1}{\eta }} > \frac{\frac{2m-1}{2(p-m)+1}\log \frac{1}{\eta }+\log
C}{\log \frac{1}{\eta }},
\end{equation*}%
for all sufficiently small $\eta >0$. Hence
\begin{equation*}
\lim\limits_{\eta \rightarrow 0^{+}}\frac{\inf \left\{ \log \left\Vert
u\right\Vert _{L_{2}\left[ 0,T\right] }:\left\Vert u-v\right\Vert _{H_{-p}%
\left[ 0,T\right] }<\eta \right\} }{\log \frac{1}{\eta }}\geq \frac{2m-1}{%
2(p-m)+1}.
\end{equation*}

To prove the converse inequality, we consider piecewise polynomial functions $%
\psi _{\eta }:[0,T]\mapsto \mathbb{R}$:
\begin{equation*}
\psi _{\eta }(t)=\left\{
\begin{array}{ll}
\sum\limits_{i=0}^{p-1}\alpha _{i}\frac{t^{p+i}}{(p+i)!\eta ^{m+i}}\qquad &
if\ t\in \lbrack 0,\eta ];\medskip \\
\frac{t^{p-m}}{(p-m)!} & if\ t>\eta ,%
\end{array}%
\right.
\end{equation*}%
and make unique choice of constants $\alpha _{0},\alpha _{1},...,\alpha
_{p-1}\in \mathbb{R}$ in such a way that $\psi _{\eta }$ becomes $(p-1)$%
-times differentiable with absolutely continuous $(p-1)^{th}$ derivative.
One can check that $\alpha =\left( \alpha _{0},\alpha _{1},...,\alpha
_{p-1}\right) $ is the unique solution of the linear system $M\alpha =b$,
where
\begin{equation*}
M=\left(
\begin{array}{ccccc}
\frac{1}{p!} & \frac{1}{(p+1)!} & \frac{1}{(p+2)!} & \cdots & \frac{1}{%
(2p-1)!}\smallskip \\
\frac{1}{(p-1)!} & \frac{1}{p!} & \frac{1}{(p+1)!} & \cdots & \frac{1}{%
(2p-2)!}\smallskip \\
\frac{1}{(p-2)!} & \frac{1}{(p-1)!} & \frac{1}{p!} & \cdots & \frac{1}{%
(2p-3)!}\smallskip \\
\vdots & \vdots & \vdots & \ddots & \vdots \smallskip \\
1 & \frac{1}{2!} & \frac{1}{3!} & \cdots & \frac{1}{p!}%
\end{array}%
\right),
\end{equation*}%
\begin{equation*}
b^{\prime}=\left( \frac{1}{(p-m)!} \ \ \frac{1}{(p-m-1)!} \ \cdots \ 1 \ \ 0
\ \cdots \ 0 \right) .
\end{equation*}
It follows that $\alpha $ does not depend on $\eta $.

Let $u=\frac{d^{p}\psi _{\eta }}{dt^{p}}$. Then,%
\begin{eqnarray*}
\left\Vert u\right\Vert _{L_{2}[0,T]}^{2} = \int_{0}^{\eta }\left(
\sum\limits_{i=0}^{p-1}\alpha _{i}\frac{t^{i}}{i!\eta ^{m+i}}\right)
^{2}\,dt=\frac{C_{1}}{\eta ^{2m-1}}, \\
\left\Vert u-v\right\Vert _{H_{-p}[0,T]}^{2} = \int_{0}^{\eta
}\!\!\left(
\sum\limits_{i=0}^{p-1}\alpha _{i}\frac{t^{p+i}}{(p+i)!\eta ^{m+i}}-\frac{%
t^{p-m}}{(p-m)!}\right) ^{2}\!\!\!\!dt= C_{2}\eta ^{2(p-m)+1},
\end{eqnarray*}%
where $C_{1},\ C_{2}$ are positive constants. This shows that there exists a
constant $C$ such that
\begin{equation*}
\begin{array}{l}
\frac{\inf \left\{ \log \left\Vert u\right\Vert _{L_{2}[0,T]}:\left\Vert
u-v\right\Vert _{H_{-p}[0,T]}<\eta \right\} }{\log \frac{1}{\eta }} \leq
\frac{\frac{2m-1}{2(p-m)+1}\log \frac{1}{\eta }+C}{\log \frac{1}{\eta }}.%
\end{array}%
\end{equation*}

\subsection{Degree of singularity for noncommutative driftless case (proof of Theorem \ref{Thdriftless})\label{dlproof}}

In the following, we consider the cost functional
\begin{equation*}
J^T\left( u\right) =\int_{0}^{T}x^{\prime }Px~dt,
\end{equation*}%
to be minimized along the trajectories of the system

\begin{equation}
\dot{x}=\sum_{i=1}^{k}g_{i}\left( x\right) u_{i}\qquad x\left( 0\right)
=x_{0}.  \label{Z1}
\end{equation}%
Let $\mathcal{A}_{x_{0}}$ denote the set of points $x\in \mathbb{R}^{n}$
which can be reached from $x_{0}$ trough trajectories of $\left( \ref{Z1}%
\right) $.

Assertion i) of Theorem \ref{Thdriftless} is obvious. Hence we only
need to prove assertions ii) and iii). We start with assertion iii).

\begin{proposition}
If $x_{T}\in \mathcal{A}_{x_{0}}$ and the quadratic form $x\mapsto x^{\prime
}Px$ admits a minimum in $\mathcal{A}_{x_{0}}$, then $\sigma _{T}\leq \frac{1}{2%
}$. $\square $
\end{proposition}

\begin{proof}
Fix $\hat{x}\in \mathcal{A}_{x_{0}}$, such that $\hat{x}^{\prime }P\hat{x}%
\leq x^{\prime }Px$ holds for all $x\in \mathcal{A}_{x_{0}}$, and consider
the controls $u_{0},u_{T}\in L_{2}^{k}[0,1]$, such that:

\begin{enumerate}
\item The trajectory generated by $u_0$ starting at $x(0)=x_0$ satisfies $%
x(1)= \hat x$;

\item The trajectory generated by $u_T$ starting at $x(0)=\hat x $ satisfies
$x(1)= x_T$.
\end{enumerate}

Then, we consider the sequence of controls%
\begin{equation*}
u_{n}\left( t\right) =nu_{0}\left( nt\right) \chi _{\left[ 0,\frac{1}{n}%
\right] }\left( t\right) +nu_{T}\left( 1-n\left( T-t\right) \right) \chi _{%
\left[ T-\frac{1}{n},T\right] }\left( t\right) .
\end{equation*}%
A simple computation shows that
\begin{equation*}
x_{nu_{0}\left( nt\right) }\left( t\right) =x_{u_{0}}\left( nt\right)
\end{equation*}%
holds for all $t\in \left[ 0,\frac{1}{n}\right] $, and
\begin{equation*}
x_{nu_{T}\left( 1-n\left( T-t\right) \right) }\left( t\right)
=x_{u_{T}}\left( 1-n\left( T-t\right) \right)
\end{equation*}%
holds for all $t\in \left[ T-\frac{1}{n},T\right] $. Therefore, for all
sufficiently large $n\in \mathbb{N}$, $u_{n}$ satisfies the boundary
condition $x_{u_{n}}\left( T\right) =x_{T}$. Now, \bigskip \newline
$\left\Vert u_{n}\right\Vert _{L_{2}[0,T]}^{2}=n^{2}\left( \int_{0}^{\frac{1%
}{n}}\left\vert u_{0}\left( nt\right) \right\vert ^{2}dt+\int_{T-\frac{1}{n}%
}^{T}\left\vert u_{T}\left( 1-n\left( T-t\right) \right) \right\vert
^{2}dt\right) =\medskip $\newline
$=n^{2}\left( \int_{0}^{1}\left\vert u_{0}\left( t\right) \right\vert ^{2}%
\frac{1}{n}dt+\int_{0}^{1}\left\vert u_{T}\left( t\right) \right\vert ^{2}%
\frac{1}{n}dt\right) =n\left( \left\Vert u_{0}\right\Vert
_{L_{2}[0,1]}^{2}+\left\Vert u_{T}\right\Vert _{L_{2}[0,1]}^{2}\right)
\bigskip $\newline
Also, \bigskip \newline
$J_{0}^{T}\left( u_{n}\right) =%
\begin{array}[t]{l}
\int_{0}^{\frac{1}{n}}x_{u_{0}}\left( nt\right) ^{\prime }Px_{u_{0}}\left(
nt\right) ~dt+\int_{\frac{1}{n}}^{T-\frac{1}{n}}\hat{x}^{\prime }P\hat{x}%
~dt+\medskip \\
+\int_{T-\frac{1}{n}}^{T}x_{u_{T}}\left( 1-n\left( T-t\right) \right)
^{\prime }Px_{u_{T}}\left( 1-n\left( T-t\right) \right) ~dt=%
\end{array}%
\medskip $\newline
$=%
\begin{array}[t]{l}
\int_{0}^{\frac{1}{n}}x_{u_{0}}\left( nt\right) ^{\prime }Px_{u_{0}}\left(
nt\right) -\hat{x}^{\prime }P\hat{x}~dt+\widehat{J_{0}^{T}}+\medskip \\
+\int_{T-\frac{1}{n}}^{T}x_{u_{T}}\left( 1-n\left( T-t\right) \right)
^{\prime }Px_{u_{T}}\left( 1-n\left( T-t\right) \right) -\hat{x}^{\prime }P%
\hat{x}~dt=%
\end{array}%
\medskip $\newline
$=\widehat{J_{0}^{T}}+\frac{1}{n}\int_{0}^{1}x_{u_{0}}\left( t\right)
^{\prime }Px_{u_{0}}\left( t\right) -\hat{x}^{\prime }P\hat{x}~dt+\frac{1}{n}%
\int_{0}^{1}x_{u_{T}}\left( t\right) ^{\prime }Px_{u_{T}}\left( t\right) -%
\hat{x}^{\prime }P\hat{x}~dt=\medskip $\newline
$=\widehat{J_{0}^{T}}+\frac{Const.}{n}, \bigskip $\newline where
$\widehat{J_{0}^{T}} = \inf \limits _u J_0^T(u)$.
This shows that, $\frac{\inf \left\{ \ln \left\Vert u\right\Vert _{L_{2}%
\left[ 0,T\right] }:J_{0}^{T}\left( u\right) \leq \widehat{J_{0}^{T}}+\frac{1%
}{n}\right\} }{\ln n}\leq \frac{\ln \sqrt{n}+Const.}{\ln n}$. By
letting $n$ go to $+ \infty $, we prove the result.
\end{proof}

\begin{proposition}
Suppose that $x_{T}\in \mathcal{A}_{x_{0}}$ and there exists $x\in \mathcal{A%
}_{x_{0}}$ such that
\begin{equation}
x^{\prime }Px<\max \left\{ x_{0}^{\prime }Px_{0},x_{T}^{\prime
}Px_{T}\right\} .  \label{Z2}
\end{equation}%
Then, $\sigma _{T}\geq \frac{1}{2}$. $\square $
\end{proposition}

\begin{proof}
The quadratic form $x\mapsto x^{\prime }Px$ admits a minimum in the closure
of $\mathcal{A}_{x_{0}}$. Let $\hat{x}\in \overline{\mathcal{A}}_{x_{0}}$ be
such a minimizer. Assumption $\left( \ref{Z2}\right) $ is equivalent to
state that $\hat{x}^{\prime }P\hat{x}<\max \left\{ x_{0}^{\prime
}Px_{0},x_{T}^{\prime }Px_{T}\right\} $. Without loss of generality , we
suppose that $\hat{x}^{\prime }P\hat{x}<x_{0}^{\prime }Px_{0}$ holds. Then,
there exist $\delta >0$, $\rho >0$ such that $\left\vert x-x_{0}\right\vert
\geq \rho $ holds whenever $x^{\prime }Px<\hat{x}^{\prime }P\hat{x}+\delta $
holds. Consider some fixed $\delta $ and $\rho $ as above. For each $%
\varepsilon \in \left] 0,\delta ^{2}\right[ $, let $u_{\varepsilon }\in
L_{2}^{k}\left[ 0,T\right] $ denote a control satisfying
\begin{equation}
J_{0}^{T}\left( u_{\varepsilon }\right) <\widehat{J_{0}^{T}}+\varepsilon .
\label{Z3}
\end{equation}
This last condition implies
\begin{equation*}
\int_{0}^{T}\left( x_{u_{\varepsilon }}^{\prime }Px_{u_{\varepsilon }}-\hat{x%
}P\hat{x}\right) ~dt<\varepsilon .
\end{equation*}%
Since $x_{u_{\varepsilon }}^{\prime }\left( t\right) Px_{u_{\varepsilon
}}\left( t\right) \geq \hat{x}P\hat{x}$ holds for all $t$, this implies
\begin{equation*}
\lambda \left\{ t\in \left[ 0,T\right] :x_{u_{\varepsilon }}^{\prime }\left(
t\right) Px_{u_{\varepsilon }}\left( t\right) -\hat{x}P\hat{x}\geq \delta
\right\} <\frac{\varepsilon }{\delta }.
\end{equation*}%
Here $\lambda $ denotes Lebesgue measure in $\mathbb{R}$. It follows that
there exists $t_{\varepsilon }\in \left] 0,\frac{\varepsilon }{\delta }%
\right[ $ such that $x_{u_{\varepsilon }}^{\prime }\left( t_{\varepsilon
}\right) Px_{u_{\varepsilon }}\left( t_{\varepsilon }\right) -\hat{x}P\hat{x}%
<\delta $. This implies that $\left\vert x_{u_{\varepsilon }}\left(
t_{\varepsilon }\right) -x_{0}\right\vert \geq \rho $. Let
\begin{eqnarray*}
\hat{t}_{\varepsilon } &=&\min \left\{ t\in \left[ 0,T\right] :\left\vert
x_{u_{\varepsilon }}\left( t_{\varepsilon }\right) -x_{0}\right\vert \geq
\rho \right\} ; \\
M &=&\max \left\{ \sum_{i=1}^{k}\left\vert g_{i}\left( x\right) \right\vert
:\left\vert x-x_{0}\right\vert \leq \rho \right\} .
\end{eqnarray*}%
It is clear that $\hat{t}_{\varepsilon }<\frac{\varepsilon }{\delta }$ and $%
M<+\infty $. Therefore, we have the estimates\bigskip \newline
$\rho =\left\vert x_{u_{\varepsilon }}\left( \hat{t}_{\varepsilon }\right)
-x_{0}\right\vert =\left\vert \int_{0}^{\hat{t}_{\varepsilon
}}\sum\limits_{i=1}^{k}g_{i}\left( x_{u_{\varepsilon }}\left( t\right)
\right) u_{i,\varepsilon }\left( t\right) ~dt\right\vert \leq \medskip
\newline
\leq \int_{0}^{\hat{t}_{\varepsilon }}\sum\limits_{i=1}^{k}\left\vert
g_{i}\left( x_{u_{\varepsilon }}\left( t\right) \right) \right\vert \times
\left\vert u_{\varepsilon }\left( t\right) \right\vert ~dt\leq \medskip $%
\newline
$\leq M\int_{0}^{\hat{t}_{\varepsilon }}\left\vert u_{\varepsilon }\left(
t\right) \right\vert ~dt\leq M\sqrt{\hat{t}_{\varepsilon }}\left\Vert
u_{\varepsilon }\right\Vert _{L_{2}\left[ 0,T\right] }\leq \frac{M\sqrt{%
\varepsilon }}{\sqrt{\delta }}\left\Vert u_{\varepsilon }\right\Vert _{L_{2}%
\left[ 0,T\right] }$.\bigskip \newline
This shows that
\begin{equation*}
\left\Vert u_{\varepsilon }\right\Vert _{L_{2}\left[ 0,T\right] }\geq \frac{%
\rho \sqrt{\delta }}{M\sqrt{\varepsilon }}=\frac{Const}{\sqrt{\varepsilon }}.
\end{equation*}%
Since $u_{\varepsilon }$ is an arbitrary control satisfying $\left( \ref{Z3}%
\right) $, it follows that
\begin{equation*}
\frac{\inf \left\{ \ln \left\Vert u\right\Vert _{L_{2}\left[ 0,T\right]
}:J_{0}^{T}\left( u\right) <\widehat{J_{0}^{T}}+\varepsilon \right\} }{\ln
\frac{1}{\varepsilon }}\geq \frac{\frac{1}{2}\ln \varepsilon +Const}{\ln
\varepsilon },
\end{equation*}%
which proves the result.
\end{proof}

\subsection{Proof of Proposition~\protect\ref{P5}\label{SS4b}}

We start with an auxiliary lemma, which establishes Lipschitz
continuity of the input-to-trajectory map of system (\ref{raf1})
with respect to so called relaxation metric in the space of
time-variant vector fields.

\begin{definition}
Let $\mathcal{O }\subset \mathbb{R}^r$ be a nonempty open set and let $%
\mathcal{F}$ be a set of time-variant vector fields $F:[0,T]\times \mathcal{%
O }\mapsto \mathbb{R}^{n}$.

$\mathcal{F}$ is said to be locally uniformly Lipschitzian with respect to $%
x $ if for every compact $K\subset \mathcal{O}$ there exists a constant $%
m<+\infty $ such that
\begin{equation*}
\left\vert F\left( t,x^{\prime }\right) -F\left( t,x\right)
\right\vert \leq m\left\vert x^{\prime }-x\right\vert
,
\end{equation*}
holds for every $ F\in \mathcal{F}$, $ t\in ] 0,T]$, $x,x^{\prime
}\in K.~\square $\footnote{ see \cite[Chapter 4]{Gam} for a more
general definition of uniformly Lipschitzian sets with $m$ depending
on $t$}

\end{definition}

\begin{lemma}
\label{P11}Consider a family of time-variant vector fields $\mathcal{F}$,
locally uniformly Lipschitzian with respect to $x$. Fix $F_{0}\in \mathcal{F}
$ and suppose that the solution of the differential equation%
\begin{equation*}
\dot{x}(t)=F_{0}\left( t,x\left( t\right) \right) ,\qquad x\left( 0\right)
=x_{0}
\end{equation*}%
(denoted by $x_{F_{0}}$) is defined for $t\in \lbrack 0,T]$. For every $F\in
\mathcal{F}$ we define the 'deviation'
\begin{equation}
\Delta _{F_{0},F}=\sup_{t\in \lbrack 0,T]}\left\vert \int_{0}^{t}F(\tau
,x_{F_{0}}(\tau ))d\tau -\int_{0}^{t}F_{0}(\tau ,x_{F_{0}}(\tau ))d\tau
\right\vert .  \label{indif}
\end{equation}%
If $\Delta _{F_{0},F}$ is sufficiently small then the solution of the
differential equation%
\begin{equation*}
\dot{x}(t)=F\left( t,x\left( t\right) \right) ,\qquad x\left( 0\right) =x_{0}
\end{equation*}%
is defined for $t\in \lbrack 0,T]$ and satisfies
\begin{equation}
\Vert x_{F}-x_{F_{0}}\Vert _{L_{\infty }[0,T]}\leq e^{mT}\Delta _{F_{0},F},
\label{gron}
\end{equation}%
where $m<+\infty $ is a constant independent of $F$. $\square $ \smallskip
\end{lemma}

\begin{remark}
Note the collocation of the norm beyond the integral sign in
(\ref{indif}). This characterizes the so called \textit{relaxation
metrics} in comparison with integral metrics. For example $\Delta
_{0,F}$ in (\ref{indif}) can become small if $F$ does not depend on
$x$ and is fast oscillating with respect to $t$ (e.g. $F\left(
t,x\right) =\cos Nt$, with $N$ large). $\square $
\end{remark}

\begin{proof}
Like in the standard proof of continuous dependence of solutions on
the right-hand side of ODE's, we can assume without loss of
generality that all fields $F\in \mathcal{F}$ vanish outside some
bounded open set containing the compact curve $x_{F_{0}}$.
Therefore, we can assume that all fields are complete and we only
need to prove that (\ref{gron}) holds.

Let $m < + \infty $ denote the Lipschitz constant of the family
$\mathcal F$. A simple computation shows that
\begin{eqnarray*}
&& \left\vert x_{F}\left( t\right) -x_{F_{0}}\left( t\right)
\right\vert = \left\vert \int_{0}^{t}F\left( \tau ,x_{F}(\tau
)\right) -F_{0}\left( \tau ,x_{F_{0}}(\tau )\right) ~d\tau
\right\vert \leq \\
&& \leq
\begin{array}[t]{l}
\left\vert \int_{0}^{t}F\left( \tau ,x_{F}(\tau )\right) -F\left( \tau
,x_{F_{0}}(\tau )\right) ~d\tau \right\vert +\medskip \\
+\left\vert \int_{0}^{t}F\left( \tau ,x_{F_{0}}(\tau )\right) -F_{0}\left(
\tau ,x_{F_{0}}(\tau )\right) ~d\tau \right\vert \leq \medskip%
\end{array}
\\
&& \leq m\int_{0}^{t}\left\vert x_{F}(\tau )-x_{F_{0}}(\tau
)\right\vert ~d\tau +\Delta _{F_{0},F}.
\end{eqnarray*}

Therefore (\ref{gron}) follows by application of Gronwall inequality.
\end{proof}

The following Lemma is a strengthened version of the well known
Gamkrelidze Approximation Lemma \cite{Gam65}, \cite[Ch. 3]{Gam}:

\begin{lemma}
\label{L3}Consider a controlled field $F\left( \cdot ,\cdot \right) :\mathbb{%
R}^{n+k}\mapsto \mathbb{R}^{n}$, continuously differentiable with respect to
all variables. Fix a compact set $K\subset \mathbb{R}^{k}$ and a relaxed
control supported in $K,$ $\eta \in \mathcal{M}_{K}\left[ 0,T\right] $, such
that $x_{\eta }$ (the trajectory of the system $\dot{x}\left( t\right)
=\left\langle F\left( x\left( t\right) ,\cdot \right) ,\eta
_{t}\right\rangle ,\quad x\left( 0\right) =x_{0}$) is defined for all $t\in %
\left[ 0,T\right] $.\newline There exists a sequence of piecewise
constant controls $\left\{ v_{N}: [0,T] \mapsto K \right\} _{N\in
\mathbb{N}}$, such that:

\begin{itemize}
\item[i)] each $v_N$ has at most $(n+2)N$ points of discontinuity;
\smallskip

\item[ii)] $\left\Vert x_{v_{N}}-x_{\eta }\right\Vert _{L_{\infty }\left[ 0,T%
\right] }=O\left( \frac{1}{N}\right) $ as $N\rightarrow \infty $. $\square $
\end{itemize}
\end{lemma}

\begin{proof}
Since $F\left( \cdot ,\cdot \right) :\mathbb{R}^{n+k}\mapsto \mathbb{R}^{n}$
is assumed to be continuously differentiable, it follows that the set
\begin{equation*}
\left\{ \left\langle F\left( \cdot ,\cdot \right) ,\nu _{t}\right\rangle
:~\nu \in \mathcal{M}_{K}\left[ 0,T\right] \right\}
\end{equation*}%
is locally uniformly Lipschitzian. Therefore, due to Lemma
\ref{P11}, we only need to show that there exists a sequence of
piecewise constant controls such that each $v_{N}$ has at most
$(n+2)N$ points of discontinuity and $\Delta _{\left\langle F,\eta
_{t}\right\rangle ,F\left( \cdot
,v_{N}\right) }=O\left( \frac{1}{N}\right) $ as $N\rightarrow \infty $.%
\newline
Fix $N>1$, and let $t_{i}=\frac{iT}{N}$, $i=0,1,2,...,N$. Due to Proposition %
\ref{P12},
\begin{equation*}
\frac{N}{T}\int_{t_{i-1}}^{t_{i}}\left\langle F\left( x_{\eta }\left(
t_{i-1}\right) ,\cdot \right) ,\eta _{t}\right\rangle ~dt\in conv\left\{
F\left( x_{\eta }\left( t_{i-1}\right) ,v\right) :v\in K\right\}
\end{equation*}%
holds and the Carath\'{e}odory theorem guarantees the existence of $%
v_{i,j}^{N}\in K$, $p_{i,j}^{N}\geq 0$ such that
\begin{eqnarray*}
&&\sum_{j=1}^{n+2}p_{i,j}^{N}=1; \\
&&\frac{T}{N}\sum_{j=1}^{n+2}p_{i,j}^{N}F\left( x_{\eta }\left(
t_{i-1}\right) ,v_{i,j}^{N}\right) =\int_{t_{i-1}}^{t_{i}}\left\langle
F\left( x_{\eta }\left( t_{i-1}\right) ,\cdot \right) ,\eta
_{t}\right\rangle ~dt.
\end{eqnarray*}%
We construct the piecewise continuous control%
\begin{equation*}
v_{N}\left( t\right) =\sum_{i=1}^{N}\sum_{j=1}^{n+2}v_{i,j}^{N}\chi _{\left[
t_{i-1}+\sum_{s=1}^{j-1}p_{i,j}^{N}\frac{T}{N},t_{i-1}+%
\sum_{s=1}^{j}p_{i,j}^{N}\frac{T}{N}\right] }\left( t\right) ,
\end{equation*}%
where $\chi _{\left[ a,b\right] }\left( t\right) $ denotes the
characteristic function of the interval $\left[ a,b\right] $. Now,
\begin{eqnarray*}
&& \hspace{-0.5cm} \Delta _{\left\langle F,\eta _{t}\right\rangle
,F\left( \cdot ,v_{N}\right) }=
\sum_{i=1}^{N}\int_{t_{i-1}}^{t_{i}}\left( F\left( x_{\eta }\left(
t\right) ,v_{N}\left( t\right) \right) -\left\langle F\left( x_{\eta
}\left( t\right) ,\cdot \right) ,\eta _{t}\right\rangle \right) ~dt= \\
&&\hspace{0.5cm}=%
\begin{array}[t]{l}
\sum\limits_{i=1}^{N}\int_{t_{i-1}}^{t_{i}}\left( F\left( x_{\eta }\left(
t\right) ,v_{N}\left( t\right) \right) -F\left( x_{\eta }\left(
t_{i-1}\right) ,v_{N}\left( t\right) \right) \right) ~dt+ \\
+\sum\limits_{i=1}^{N}\int_{t_{i-1}}^{t_{i}}\left( F\left( x_{\eta }\left(
t_{i-1}\right) ,v_{N}\left( t\right) \right) -\left\langle F\left( x_{\eta
}\left( t_{i-1}\right) ,\cdot \right) ,\eta _{t}\right\rangle \right) ~dt+
\\
+\sum\limits_{i=1}^{N}\int_{t_{i-1}}^{t_{i}}\left( \left\langle F\left(
x_{\eta }\left( t_{i-1}\right) ,\cdot \right) ,\eta _{t}\right\rangle
-\left\langle F\left( x_{\eta }\left( t\right) ,\cdot \right) ,\eta
_{t}\right\rangle \right) ~dt.%
\end{array}%
\end{eqnarray*}%
Since all fields being considered form a locally uniformly Lipschitzian set
and $x_{\eta }$ is Lipschitzian with respect to time, there exists a
constant $L<+\infty $, independent of $N$, such that
\begin{eqnarray*}
&&\left\vert \sum\limits_{i=1}^{N}\int_{t_{i-1}}^{t_{i}}\left( F\left(
x_{\eta }\left( t\right) ,v_{N}\left( t\right) \right) -F\left( x_{\eta
}\left( t_{i-1}\right) ,v_{N}\left( t\right) \right) \right) ~dt\right\vert
\leq  \\
&\leq &\sum\limits_{i=1}^{N}\int_{t_{i-1}}^{t_{i}}L\left\vert x_{\eta
}\left( t\right) -x_{\eta }\left( t_{i-1}\right) \right\vert ~dt\leq  \\
&\leq &\sum\limits_{i=1}^{N}\int_{t_{i-1}}^{t_{i}}L^{2}\left(
t-t_{i-1}\right) ~dt=\frac{L^{2}T^{2}}{2N}
\end{eqnarray*}%
holds for every sufficiently large $N$. The same argument gives the similar
inequality
\begin{equation*}
\sum\limits_{i=1}^{N}\int_{t_{i-1}}^{t_{i}}\left( \left\langle F\left(
x_{\eta }\left( t_{i-1}\right) ,\cdot \right) ,\eta _{t}\right\rangle
-\left\langle F\left( x_{\eta }\left( t\right) ,\cdot \right) ,\eta
_{t}\right\rangle \right) ~dt\leq \frac{L^{2}T^{2}}{2N}.
\end{equation*}%
Finally,
\begin{eqnarray*}
&&\sum\limits_{i=1}^{N}\int_{t_{i-1}}^{t_{i}}\left( F\left( x_{\eta }\left(
t_{i-1}\right) ,v_{N}\left( t\right) \right) -\left\langle F\left( x_{\eta
}\left( t_{i-1}\right) ,\cdot \right) ,\eta _{t}\right\rangle \right) ~dt= \\
&=&%
\begin{array}[t]{l}
\sum\limits_{i=1}^{N}\sum\limits_{j=1}^{n+2}\int_{\left(
t_{i-1}+\sum_{s=1}^{j-1}p_{i,j}^{N}\frac{T}{N}\right) }^{\left(
t_{i-1}+\sum_{s=1}^{j}p_{i,j}^{N}\frac{T}{N}\right) }F\left( x_{\eta }\left(
t_{i-1}\right) ,v_{i,j}^{N}\right) ~dt-\medskip  \\
\hspace{1cm}
-\sum\limits_{i=1}^{N}\int_{t_{i-1}}^{t_{i}}\left\langle F\left(
x_{\eta
}\left( t_{i-1}\right) ,\cdot \right) ,\eta _{t}\right\rangle ~dt=%
\end{array}
\\
&=&\sum\limits_{i=1}^{N}\left( \frac{T}{N}\sum%
\limits_{j=1}^{n+2}p_{i,j}^{N}F\left( x_{\eta }\left( t_{i-1}\right)
,v_{i,j}^{N}\right) -\int_{t_{i-1}}^{t_{i}}\left\langle F\left( x_{\eta
}\left( t_{i-1}\right) ,\cdot \right) ,\eta _{t}\right\rangle ~dt\right) =0,
\end{eqnarray*}%
which proves that $v_{N}\left( t\right) $ has the desired property.
\end{proof}

For the second approximation step we will use the following lemma concerning
approximation of piece-wise continuous controls by absolutely continuous
controls.

\begin{lemma}
\label{L4}Consider a controlled field $F\left( \cdot ,\cdot \right)
:\mathbb{ R}^{n+k}\mapsto \mathbb{R}^{n}$, continuously
differentiable with respect to all variables. Fix compact sets $K_1
\subset \mathbb R^n$, $K_2 \subset \mathbb R^k$.

For every piecewise constant control $v :[0,T] \mapsto K_2$ with $N$
points of discontinuity such that $x_v$ (the trajectory of the
system $\dot x =F(x,v)$, $x(0)=x_0$) lies in $K_1$ and every
sufficiently small $\varepsilon >0$ there exists a continuous
piecewise linear control $w_\varepsilon : [0,T] \mapsto {\rm
conv}\left( \{ 0 \} \cup K_2 \right)$ such that:

\begin{itemize}
\item[i)] $\left\| x_{w_\varepsilon}-x_{v}\right\| _{L_{\infty }\left[ 0,T
\right] } \leq CN\varepsilon $; \smallskip

\item[ii)] $\left\| \dot w _\varepsilon \right\| \leq \frac{CN} \varepsilon $.
\end{itemize}
Here $x_{w_\varepsilon}$ denotes the trajectory of the system $\dot
x =F(x,v)$, $x(0)=x_0$ and $C<+\infty$ is a constant depending only
on the sets $K_1$, $K_2$. $\square $
\end{lemma}

\begin{proof}
Fix $v$ satisfying the assumptions of the Lemma. Let
$0=t_0<t_1<t_2<...<t_N=T$ be the points of discontinuity of $v$. $v$
can be represented as \[v(t)= \sum_{i=1}^N v_i \chi _{[t_{i-1}, t_i
[}(t). \] Fix a small $\varepsilon >0 $ and let
\begin{eqnarray*}
&& i_0 = 0, \qquad v_{i_0}=0, \\
&& i_j= \left\{
\begin{array}{ll}
\min \left\{ i: t_i \geq t_{i_{j-1}} + \varepsilon \right\}, \quad &
{\rm if } \ \left\{ i: t_i \geq t_{i_{j-1}} + \varepsilon \right\}
\neq \emptyset, \medskip \\
N, & {\rm if } \ \left\{ i: t_i \geq t_{i_{j-1}} + \varepsilon
\right\} = \emptyset .
\end{array}
\right.
\end{eqnarray*}
The piecewise linear control \[ w_\varepsilon (t) = \frac 1
\varepsilon \sum _{j=1} ^N \left( (v_{i_j}-v_{i_{j-1}})t \chi
_{[t_{i_{j-1}},t_{i_{j-1}}+\varepsilon[}(t) + v_{i_j}\chi
_{[t_{i_{j-1}}+ \varepsilon,t_{i_j}[}(t) \right)
\]
takes values on ${\rm conv}\left( \{ 0 \} \cup K_2 \right) $ and
differs from $v$ only on the union of intervals $ \bigcup\limits
_{j=1}^N [t_{i_{j-1}},t_{i_j}[ $. Since $K_2$ is compact, there
exists a constant $C_1<+ \infty $ such that $|
v_{i_j}-v_{i_{j-1}}|^2 < C_1$ holds for $j=1,2,...,N$. Therefore, \[
\| \dot w_\varepsilon \|^2_{L_2[0,T]}=\sum\limits_{j=1}^N \frac {|
v_{i_j}-v_{i_{j-1}}|^2 }{\varepsilon ^2} \varepsilon \leq \frac{C_1
N} \varepsilon.
\]
The Lemma \ref{P11} guarantees that the inequality \[ \left\|
x_{w_\varepsilon} -x_v \right\|_{L_\infty [0,T]} \leq e^{mT}
\sup\limits_{t \in[0,T]} \left| \int _0^t F(x_v,w_\varepsilon ) -
F(x_v,v) \, d \tau \right|
\]
holds provided the right-hand side is sufficiently small.

Since $x_v$ lies in $K_1$, $v$ lies in $K_2$ and $w_\varepsilon$
lies in ${\rm conv} \left( \{ 0 \} \cup K_2 \right) $, there exists
a constant $C_2 < + \infty $ such that
\begin{eqnarray*}
&& \hspace{-0.5cm} \sup\limits_{t \in[0,T]} \left| \int _0^t
F(x_v,w_\varepsilon ) -
F(x_v,v) \, d \tau \right| \leq \\
&& \hspace{0.5cm} \leq \sum_{j=1}^N \int _{t_{i_{j-1}}}
^{t_{i_{j-1}}+\varepsilon}\left| F(x_v,w_\varepsilon ) - F(x_v,v) \,
d \tau \right| \leq N C_2 \varepsilon .
\end{eqnarray*}
Therefore the Lemma holds for $C \geq \max \{ C_1, e^{mT}C_2 \}$.
\end{proof}

To conclude the proof of Proposition \ref{P5}, we consider the
augmented state $z(t)=\left( J_r^t, \, y(t) \right)$ with dynamics
\begin{equation}
\dot z (t) = \sum_{i=1}^{n+2} p_i(t)F(z(t),v^i(t) ) =
\sum_{i=1}^{n+2} p_i(t) \left( \begin{array}[c]{c} \tilde f_0 (y(t),
v^i(t) ) \smallskip \\ \tilde f (y(t), v^i(t) )
\end{array}\right) , \label{Z5}
\end{equation}
with  $ \left(\tilde f_0 , \, \tilde f  \right)$ defined by
(\ref{laf}).

Under the assumptions of Proposition \ref{P5}, there exists a
compact set $K \subset \mathbb R^k $ such that \[ \eta = \sum
_{i=1}^{n+2} p_i \delta_{v^i} \in \mathcal M _K [0,T].
\]
Let $N _\varepsilon =O\left( \frac 1 \varepsilon \right) $ when
$\varepsilon \rightarrow 0^+$. Due to Lemma \ref{L3} there exist
piecewise constant controls $\left\{ w_\varepsilon : [0,T] \mapsto K
\right\}_{\varepsilon
>0} $ such that $v_\varepsilon $ has $O\left( \frac 1 \varepsilon
\right) $ points of discontinuity and the corresponding trajectories
of (\ref{Z5}) satisfy \[ \left\| z_{w_\varepsilon} -z_\eta
\right\|_{L_\infty [0,T]} = O(\varepsilon), \qquad {\rm when} \
\varepsilon \rightarrow 0^+.
\]
Due to Lemma \ref{L4},there exist continuous piecewise linear
controls \linebreak $\left\{ v_\varepsilon : [0,T] \mapsto {\rm
conv}\left(\{0\} \cup K \right) \right\}_{\varepsilon
>0} $ such that
\[ \left\| z_{v_\varepsilon} -z_{w_\varepsilon}
\right\|_{L_\infty [0,T]} = O(\varepsilon), \qquad \left\| \dot
v_\varepsilon \right\|_{L_2 [0,T]} = O(\varepsilon^{-3}),
\]
when $ \varepsilon \rightarrow 0^+$. It follows that \[ \left\|
z_{v_\varepsilon} -z_\eta \right\|_{L_\infty [0,T]} \leq \left\|
z_{v_\varepsilon} -z_{w_\varepsilon} \right\|_{L_\infty [0,T]}
+\left\| z_{w_\varepsilon} -z_\eta \right\|_{L_\infty [0,T]}
=O(\varepsilon).
\]
This shows that conditions i), ii) and iv) of Proposition \ref{P5}
can be satisfied.

$y_{v_\varepsilon} (T) $ lies $\varepsilon$-close to $y_\eta (T)$,
which lies in the integral manifold of $G$ that contains $x_T$.
Therefore, $v_\varepsilon$ can be modified (without changing the
magnitude of the estimates above) in such a way that \[ \left| x_T -
e^{Gv_\varepsilon (T)} y_{v_\varepsilon }(T) \right|=O(\varepsilon
), \qquad {\rm when } \ \varepsilon \rightarrow 0^+.
\]
This concludes the proof.

\end{document}